\let\ORIlabel\label
\let\ORIrefstepcounter\refstepcounter
\AddToHook{package/hyperref/before}{%
  \let\label\ORIlabel
  \let\refstepcounter\ORIrefstepcounter
}

\documentclass[final,onefignum,onetabnum]{siamart171218}

\usepackage{amsfonts}
\usepackage{amssymb}
\usepackage{amsmath}
\usepackage{mathtools}
\usepackage{graphicx}
\usepackage{algorithmic}
\usepackage{booktabs}
\usepackage{array}
\usepackage{bm}
\usepackage{url}
\usepackage{xcolor}
\usepackage{enumitem}
\usepackage{microtype}
\usepackage{tabularx}
\usepackage{placeins}
\usepackage{tikz}

\ifpdf
  \DeclareGraphicsExtensions{.pdf,.png,.jpg}
\else
  \DeclareGraphicsExtensions{.eps}
\fi

\newsiamremark{remark}{Remark}
\newsiamremark{hypothesis}{Hypothesis}
\newsiamremark{discussion}{Discussion}
\newsiamremark{example}{Example}
\crefname{hypothesis}{Hypothesis}{Hypotheses}
\crefname{example}{Example}{Examples}
\newsiamthm{claim}{Claim}
\newsiamthm{assumption}{Assumption}

\DeclareMathOperator{\diag}{diag}
\DeclareMathOperator{\dom}{dom}

\DeclareMathOperator{\gra}{gra}
\DeclareMathOperator{\Fix}{Fix}
\DeclareMathOperator{\zer}{zer}
\DeclareMathOperator{\Lip}{Lip}
\DeclareMathOperator{\prox}{prox}

\DeclareMathOperator{\proj}{proj}
\DeclareMathOperator{\argmin}{argmin}

\DeclareMathOperator{\RePart}{Re}

\newcommand{\R}{\mathbb{R}}

\newcommand{\cX}{\mathcal{X}}

\newcommand{\Rplus}{\mathbb{R}_{+}}
\newcommand{\Rpp}{\mathbb{R}_{++}}
\newcommand{\one}{\mathbf{1}}

\newcommand{\T}{^{\mathsf T}}

\newcommand{\inner}[2]{\left\langle #1,#2\right\rangle}
\newcommand{\norm}[1]{\left\lVert #1\right\rVert}

\newcommand{\paren}[1]{\left( #1\right)}

\setlist[itemize]{leftmargin=*,topsep=3pt,itemsep=2pt}
\setlist[enumerate]{leftmargin=*,topsep=3pt,itemsep=2pt}

\headers{Splitting Compositions of Monotone Operators}{H. Qu, Y. Wu, and T. Chaffey}

\ifpdf
\hypersetup{
  pdftitle={Triangular Splitting for Compositions of Monotone Operators},
  pdfauthor={Haoming Qu, Yefan Wu, and Thomas Chaffey},
  pdfsubject={Kernel-designed triangular splitting, gain-matrix convergence, and applications of coupled monotone-operator compositions},
  pdfkeywords={monotone operator, operator composition, nonmonotone inclusion, warped resolvent, triangular kernel, small-gain theorem, saddle system, nonlinear capacitor, equilibrium network}
}
\fi

\begin{document}

\makeatletter
\@ifundefined{cref@old@refstepcounter}{}{%
  \def\refstepcounter@optarg[#1]#2{%
    \cref@old@refstepcounter{#2}%
    \cref@constructprefix{#2}{\cref@result}%
    \@ifundefined{cref@#1@alias}%
      {\def\@tempa{#1}}%
      {\def\@tempa{\csname cref@#1@alias\endcsname}}%
    \protected@edef\cref@currentlabel{%
      [\@tempa][\arabic{#2}][\cref@result]%
      \csname p@#2\endcsname\csname the#2\endcsname}}%
}
\makeatother

\title{Splitting Compositions of Monotone Operators}

\author{
Haoming Qu\thanks{School of Electrical and Computer Engineering, University of Sydney, NSW 2008, Australia.}
\and
Yefan Wu\thanks{School of Mathematics and Statistics, University of Sydney, NSW 2006, Australia.}
\and
Thomas Chaffey\footnotemark[1]
}

\thanks{A preliminary extended abstract of parts of this work was accepted for presentation at the 2026 International Symposium on Mathematical Theory of Networks and Systems (MTNS 2026) under the title ``A Proximal Point Algorithm for the Composition of Monotone Operators.''}

\maketitle

\begin{abstract}
Coupled convex optimisation equilibria, bilinear saddle and potential systems, and nonlinear electronic device laws lead naturally to inclusions of the form $0\in Hx+A(Bx)-b$, where $A$ and $B$ are maximal monotone but their composition need not be monotone. We introduce a new class of splitting algorithms that solve such inclusions. Introducing the intermediate graph variable $y\in Bx$ exposes a directed two-row system. A one-parameter family of kernels produces a lower-triangular warped-resolvent iteration evaluated with one application of $H$ and the ordinary resolvents $J_{\alpha A}$ and $J_{\beta B}$. Appropriate specializations of the problem data and algorithmic parameters recover forward--backward, inertial, primal--dual, and Douglas--Rachford constructions, while setting $H=b=0$ generates three novel algorithms for the pure composition $0\in A(Bx)$.

Convergence of the coupled method is certified by an explicit nonnegative two-block gain matrix. Its Schur condition yields contraction in particular weighted norms, global linear convergence, and residual error estimates. For the general problem class, our convergence condition is shown to be sharp for an explicitly constructed linear family. When the problem is specialized to a pure composition, we prove convergence under the weaker assumption that only one operator is strongly monotone. This weaker regime is illustrated on a nonlinear circuit problem. We further illustrate our methods by solving inference problems for cascaded equilibrium neural networks.
\end{abstract}



\section{Introduction}

Monotone inclusions provide a common language for convex optimization, variational inequalities, equilibrium computation, and passive network models. The proximal-point method represents zeros of a maximal monotone operator as fixed points of its resolvent, Douglas--Rachford splitting exposes the resolvents of two operators in a monotone sum, and forward--backward and primal--dual schemes exploit smooth terms, inverse relations, and linearly coupled product-space formulations \cite{Rockafellar1976,LionsMercier1979,EcksteinBertsekas1992,BauschkeCombettes2017,BricenoAriasCombettes2011,ChambollePock2011,Condat2013,Vu2013,CombettesPesquet2012}. These constructions are effective because the splitting preserves an operator geometry for which the required resolvents remain accessible.

We study inclusions of the form
\begin{equation}
\label{eq:intro-master}
0\in Hx+A(Bx)-b,
\qquad
A(Bx):=\bigcup_{y\in Bx}Ay,
\end{equation}
where $A,B:\mathcal X\rightrightarrows\mathcal X$ are maximal monotone on a real Hilbert space $\mathcal X$, $H:\mathcal X\to\mathcal X$ is single-valued, and $b\in\mathcal X$. Maximal monotonicity of the factors does not in general pass to $A\circ B$, meaning established monotone operator splitting methods do not apply.  Furthermore, the resolvent of the complete composition is generally unavailable even when $J_{\alpha A}$ and $J_{\beta B}$ are inexpensive.

\subsection{Relation to existing splitting frameworks}

Splitting algorithms solve an inclusion involving multiple operators using computation steps that each involve only a single operator.  The computational treatment of an operator is determined by the form in which that operator is accessible. Linear-composite splitting accesses an inner map through a linear operator and its adjoint, as in $L^*C(Lx)$ and related monotone--skew lifts \cite{BricenoAriasCombettes2011,CombettesPesquet2012,Condat2013,Vu2013}. Smooth nonlinear primal--dual and prox-linear methods instead use derivative information or a local linearization of the inner mapping \cite{Valkonen2014,ClasonMazurenkoValkonen2019,LewisWright2016,DrusvyatskiyPaquette2019}. In the present setting, $B$ may be nonlinear and set-valued, so its graph relation $y\in Bx$ and its resolvent $J_{\beta B}$ provide the natural  access model.

Once the available operator evaluations have been fixed, splitting design concerns how they are assembled into an implementable coupled step that converges to a solution of the original inclusion. Prediction--correction and asymmetric forward--backward--adjoint (AFBA) splitting organize coupled updates through correction operators or non-self-adjoint preconditioners \cite{he2012convergence,LatafatPatrinos2017}. This asymmetric-preconditioning principle also informs the kernel construction developed below. Degenerate preconditioned proximal methods and warped-resolvent theory place this auxiliary operator at the level of the implicit equation, where its interaction with the lifted operator determines whether the step decomposes into tractable block evaluations \cite{BrediesEtAl2022,BuiCombettes2020}. In warped-resolvent terminology, this auxiliary single-valued operator is the kernel. Frugal and graph-based splitting frameworks express the resulting computational structure through dependencies among direct, forward, and resolvent evaluations \cite{morin2024frugal,bredies2024graph,aragon2025forward}. These viewpoints describe a common design chain: the auxiliary operator determines the block factorization of the implicit problem, and that factorization determines the order in which the available operator primitives are evaluated.

For \eqref{eq:intro-master}, this design chain is realized directly by the composition lift. Retaining $y\in Bx$ produces the two-row lift $b\in Ay+Hx$ and $y\in Bx$, whose directed coupling records how the output associated with one factor enters the other. We choose the kernel so that its sum with the lifted operator cancels the upper-right coupling and becomes block lower triangular. This leads to the following evaluation order: one application of $H$ enters the update through $J_{\alpha A}$, and the resulting intermediate state is then transmitted to $J_{\beta B}$. It also separates the Lipschitz gains of $H$, $J_{\alpha A}$, and $J_{\beta B}$, which become the entries of the comparison matrix used in the convergence analysis. The resulting factor-level description expresses the regularity of the lifted nonmonotone problem through its accessible components, within the broader proximal theory of complete operators and operator--kernel pairs \cite{EvensPasLatafatPatrinos2025,LeDaoThera2025}.

In contrast to linear-composite and derivative-based formulations, the inner factor $B$ may be nonlinear and set-valued and is accessed through its resolvent. Within the warped-resolvent framework, we construct kernels specifically for the directed lift of $H+A\circ B-b$, so each step separates into the ordinary resolvents of $A$ and $B$. Convergence is established from the separate component gains, including cases in which the lift is nonmonotone, rather than from monotonicity of the complete lifted operator.

\subsection{Directed lifting and kernel design}

Introducing $y\in Bx$ in \eqref{eq:intro-master} gives the equivalent system
\begin{equation}
\label{eq:intro-lift}
b\in Ay+Hx,
\qquad
y\in Bx.
\end{equation}
The two rows form a directed coupling: $x$ enters the $A$-relation through $H$, while the graph relation $y\in Bx$ links the intermediate variable $y$ back to $x$. A composition-adapted kernel cancels the upper-right coupling in the corresponding implicit equation, makes the equation block lower triangular, and yields
\begin{equation}
\label{eq:intro-algorithm}
y^{k+1}=J_{\alpha A}(y^k-\alpha Hx^k+\alpha b),
\qquad
\bar y^k=(1+\theta)y^{k+1}-\theta y^k,
\qquad
x^{k+1}=J_{\beta B}(x^k+\beta\bar y^k),
\end{equation}
with $\alpha,\beta>0$ and $\theta\in[-1,1]$. The values $\theta=-1,0,1$ transmit the previous, newly computed, and reflected intermediate states. The inclusion data $(A,B,H,b)$ determine the fixed-point set, while $(\alpha,\beta,\theta)$ and the kernel determine its algorithmic realization.

This separation places several established methods and a new pure-composition branch in one family. Exact choices of the problem data recover forward--backward, inertial, primal--dual, and Douglas--Rachford/ADMM iterations. Setting $H=b=0$ acts directly on $0\in A(Bx)$ and produces a new family of algorithms using only the two ordinary resolvents. The pure-composition branch is presented at the end of Section~\ref{sec:framework}, after the classical specializations.

\subsection{Convergence mechanism and consequences}

Convergence must be established for the iteration~\eqref{eq:intro-algorithm} without relying on standard techniques for monotone operators. We use componentwise estimates for each entry of the triangular map to produce a nonnegative $2\times2$ gain matrix. Schur stability gives constructive Perron-weighted and diagonal Hilbert product contractions, global linear convergence, residual error estimates, and a scalar feasibility criterion.  For $\theta \in [-1, 0]$, the convergence certificate is shown to be tight.

For $H=b=0$, the first row becomes autonomous. In this case, we are able to prove convergence under the weaker assumption that only one operator is strongly monotone and the other is merely monotone. A nonlinear varicap is an electronic example of such a composition. Monotone operator equilibrium networks provide a further application setting \cite{WinstonKolter2020,Chaffey2025}. 

\subsection{Contributions and organization}

The principal contributions are as follows.
\begin{enumerate}
\item We identify a broad problem class represented by \eqref{eq:intro-master}, with explicit operator identifications for coupled convex optimisation equilibria, bilinear saddle systems, bilinear potential minimization, and nonlinear circuits.
\item We show the problem can be lifted to a product space and construct a one-parameter family of triangular kernels for this lift. These kernels give a warped resolvent iteration that accesses $A$ and $B$ only through their resolvents $J_{\alpha A}$ and $J_{\beta B}$.
\item We recover representative classical methods and derive a new family of algorithms for the inclusion $0\in Hx+A(Bx)-b$.
\item We prove a two-block gain-matrix convergence theorem, an exact scalar Schur criterion, and explicitly construct examples that show necessity of these conditions.
\item We establish convergence for pure compositions under joint resolvent contractivity or strong monotonicity of only one factor. We illustrate the latter regime on periodic varicap voltage recovery and extend the lower-triangular mechanism to finite MonDEQ cascades.
\end{enumerate}

Section~\ref{sec:problem-classes} derives the problem classes leading to \eqref{eq:intro-master}. Section~\ref{sec:framework} develops the kernel family, its classical specializations, and the new pure-composition algorithms. Section~\ref{sec:certificate-calibration} gives the gain-matrix convergence and tightness analysis. Section~\ref{sec:pure-composition} treats pure compositions under weaker assumptions and presents the varicap example. Section~\ref{sec:mondeq} develops the MonDEQ cascade. Appendices~\ref{app:gain}--\ref{app:numerics} record reusable block-gain constructions and numerical protocols.
\section{Representative problem classes}
\label{sec:problem-classes}
\label{sec:convex-games}

In this section, we show that the problem $0 \in Hx+A(Bx)$ arises naturally from a range of problems: coupled convex optimisation problems, their bilinear saddle and minimization subclasses, and nonlinear electronic device laws. 

\subsection{Coupled convex optimisation problems}
\label{subsec:convex-game-representation}

Let $\cX$ be a real Hilbert space, let $f,g\in\Gamma_0(\cX)$, let $H:\cX\to\cX$ be single-valued, and let $b\in\cX$. Define
\begin{equation}
\label{eq:convex-game-objectives}
\mathcal J_g(y;x):=g(y)+\inner{Hx-b}{y},\qquad \mathcal J_f(x;y):=f(x)-\inner{y}{x},
\end{equation}
and
\begin{equation}
\label{eq:convex-response-maps}
\mathcal R_g(x):=\argmin_v\mathcal J_g(v;x),\qquad \mathcal R_f(y):=\argmin_u\mathcal J_f(u;y).
\end{equation}
Each block therefore solves a convex optimisation problem in its own variable. The resulting equilibrium formulation is consistent with the variational-inequality treatment of convex games and monotone Nash systems \cite{Rosen1965,FacchineiPang2003,BricenoAriasCombettes2013,YiPavel2019}.

\begin{proposition}[Coupled convex equilibrium representation]
\label[proposition]{prop:coupled-response}
A pair $(y^\star,x^\star)$ satisfies
\begin{equation}
\label{eq:coupled-response-optimization}
y^\star\in\mathcal R_g(x^\star),\qquad x^\star\in\mathcal R_f(y^\star)
\end{equation}
if and only if
\begin{equation}
\label{eq:coupled-response-relations}
b-Hx^\star\in\partial g(y^\star),\qquad y^\star\in\partial f(x^\star).
\end{equation}
Consequently, $x^\star$ satisfies
\begin{equation}
\label{eq:subdiff-master-inclusion}
0\in Hx^\star+(\partial g\circ\partial f)(x^\star)-b.
\end{equation}
\end{proposition}

\begin{proof}
Fermat's rule converts \eqref{eq:coupled-response-optimization} into $0\in\partial g(y^\star)+Hx^\star-b$ and $0\in\partial f(x^\star)-y^\star$, which are precisely \eqref{eq:coupled-response-relations}. The final statement is the definition of the set-valued composition.
\end{proof}

For general $H$, the cross-coupling need not satisfy an integrability condition, and \eqref{eq:coupled-response-relations} defines a coupled convex optimisation system. Two scalar choices of $H$ yield canonical variational subclasses.
\subsection{Bilinear saddle and bilinear minimization models}
\label{subsec:integrable-boundaries}

Bilinear saddle point problems are a class of convex--concave optimization problems that arise in a wide range of machine learning problems \cite{dirren2025contractivity}.  The following proposition shows that the weighted saddle model is included within our composition framework. In finite dimensions, it is a scalar-coupling instance of the bilinear saddle-point problem studied by Dirren et al.~\cite{dirren2025contractivity}, who exploit the saddle-point structure to establish strict contractivity of several primal--dual methods, including Chambolle--Pock. Connections with classical primal--dual methods are discussed in \Cref{subsec:classical-specializations}.

\begin{proposition}[Weighted bilinear saddle model]
\label[proposition]{prop:zero-sum-boundary}

Let $H=\gamma I$ with $\gamma>0$. Then \eqref{eq:coupled-response-relations} is the saddle system
\begin{equation}
\label{eq:zero-sum-saddle}
\min_y\max_x\ \mathcal L(y,x),
\qquad
\mathcal L(y,x):=g(y)-\inner{b}{y}+\gamma\inner{y}{x}-\gamma f(x).
\end{equation}
\end{proposition}

\begin{proof}
The saddle conditions are $0\in\partial g(y)-b+\gamma x$ and $0\in\gamma\partial f(x)-\gamma y$. Dividing the second relation by $\gamma$ gives \eqref{eq:coupled-response-relations}.
\end{proof}
When $f$ is strongly convex, the second relation uniquely determines $x=\nabla f^*(y)$, and the same system reduces to an ordinary convex problem through Fenchel conjugacy~\cite{BauschkeCombettes2017}.

The opposite coupling sign yields a joint bilinear minimization problem.

\begin{proposition}[Weighted bilinear potential minimization]
\label[proposition]{prop:potential-boundary}
Let $H=-\gamma I$ with $\gamma>0$ and define
\begin{equation}
\label{eq:weighted-potential}
\mathcal P(y,x):=g(y)-\inner{b}{y}+\gamma f(x)-\gamma\inner{y}{x}.
\end{equation}
The block first-order conditions of $\mathcal P$ are exactly \eqref{eq:coupled-response-relations}. Whenever $\mathcal P$ is convex, these relations characterize its minimizers; the curvature condition in Section~\ref{sec:certificate-calibration} gives uniqueness.
\end{proposition}

\begin{proof}
The block conditions are $0\in\partial g(y)-b-\gamma x$ and $0\in\gamma\partial f(x)-\gamma y$, which give \eqref{eq:coupled-response-relations} with $H=-\gamma I$.
\end{proof}

Thus the same composition architecture contains a zero-sum saddle boundary for $H=\gamma I$, a potential minimization boundary for $H=-\gamma I$, and a general coupled convex optimisation system for nonscalar or nonintegrable $H$.

\subsection{Nonlinear circuit elements}
\label{subsec:varicap-model}

The same composition architecture also arises naturally in dynamic physical systems, where a static constitutive relation is followed by a dynamic operator. For a nonlinear capacitor, the voltage-to-charge relation is static, while temporal differentiation maps charge to current, producing a dynamic--static composition, which is not, in general, monotone \cite{ChaffeySepulchre2024, kulkarni2002incremental}. Monotone circuit relations originate in network theory \cite{Minty1960,ChuaDesoerKuh1987} and have recently been combined with monotone--skew splitting for circuit computation \cite{ChaffeyBanertGiselssonPates2023,ChaffeySepulchre2024}. At the single-device level, we consider periodic voltages and currents with period $T$ and work in $\cX=L^2(0,T)$. The differentiation operator $d/dt$ is defined on $\{q\in H^1(0,T):q(0)=q(T)\}$, imposing periodic boundary conditions. It is monotone \cite[Example~20.9]{BauschkeCombettes2017} and its kernel consists of constant functions.

As a specific example, we consider a varicap nonlinear capacitor.  For $K,V_0,\varepsilon>0$, define the regularized varicap law pointwise by
\begin{equation}
\label{eq:single-varicap-law}
B_\varepsilon(v):=
\begin{cases}
-K\sqrt{V_0-v}+\varepsilon v,&v<V_0,\\
[\varepsilon V_0,\infty),&v=V_0.
\end{cases}
\end{equation}
This is a maximal $\varepsilon$-strongly monotone charge--voltage relation. Given a prescribed zero-mean periodic current $i^\star\in\cX$, voltage recovery requires a charge $q$ in the domain of $d/dt$ such that
\begin{equation}
\label{eq:varicap-composition-model}
q(t)\in B_\varepsilon(v(t))\quad\text{a.e.},\qquad \frac{dq}{dt}=i^\star.
\end{equation}
With
\begin{equation}
\label{eq:varicap-operator-dictionary}
A(q):=\frac{dq}{dt}-i^\star,\qquad B:=B_\varepsilon,
\end{equation}
this is the pure composition
\begin{equation}
\label{eq:varicap-pure-composition}
0\in A(B(v)),
\end{equation}
namely the case $H=0$ and $b=0$. The mean charge is restricted to the affine subset 
\begin{equation}
\label{eq:charge-mean-slice}
\mathcal Q_{\rm ref}:=\left\{q\in L^2(0,T):\frac{1}{T}\int_0^Tq(t)\,dt=\bar q_{\rm ref}\right\},
\end{equation}
which fixes the DC offset without changing the physical current. Section~\ref{sec:pure-composition} returns to this example after the pure-composition convergence theory.

\subsection{The common $H+A\circ B$ architecture}
\label{subsec:common-problem}

The preceding models have the same operator representation. Let $A,B:\cX\rightrightarrows\cX$ be maximal monotone, let $H:\cX\to\cX$ be single-valued, and let $b\in\cX$. We study the zero-finding problem
\begin{equation}
\label{eq:master-problem}
0\in Hx+A(Bx)-b.
\end{equation}
Introducing the intermediate variable $y\in Bx$ gives
\begin{equation}
\label{eq:master-lifted}
b\in Ay+Hx,
\qquad
y\in Bx.
\end{equation}

\begin{table}[t]
\centering
\caption{Representative problem classes and their operator identifications.}
\label{tab:problem-dictionary}
\small
\renewcommand{\arraystretch}{1.10}
\begin{tabularx}{0.96\textwidth}{@{}l c c c X@{}}
\toprule
problem class & $A$ & $B$ & $H$ & resulting model \\
\midrule
coupled convex optimisation equilibria & $\partial g$ & $\partial f$ & general & coupled optimisation system \eqref{eq:coupled-response-relations} \\
bilinear saddle & $\partial g$ & $\partial f$ & $\gamma I$ & saddle problem \eqref{eq:zero-sum-saddle} \\
bilinear minimization & $\partial g$ & $\partial f$ & $-\gamma I$ & potential \eqref{eq:weighted-potential} \\
periodic varicap & $\frac{d}{dt}-i^\star$ & $B_\varepsilon$ & $0$ & $i^\star\in\frac{d}{dt}B_\varepsilon(v)$ \\
\bottomrule
\end{tabularx}
\end{table}

The lifted system retains an asymmetric directed dependence: the first row updates $y$ from $x$ through $H$, whereas the second row recovers $x$ through the graph of $B$ from the intermediate variable $y$. The algorithmic task is therefore to expose the ordinary resolvents of $A$ and $B$ while preserving this directed structure. Section~\ref{sec:framework} constructs kernels that achieve this by an exact triangular factorization.

\section{Triangular splitting for compositions}
\label{sec:framework}
\label{sec:algorithm-family}

Let $J_{\alpha A}:=(I+\alpha A)^{-1}$ denote the resolvent of a maximal monotone operator $A$. By Minty's theorem, $J_{\alpha A}$ is single-valued and has full domain. The construction below uses the warped-resolvent framework \cite{BuiCombettes2020} to turn the directed lift in \eqref{eq:master-lifted} into a lower-triangular map evaluated with the ordinary resolvents of $A$ and $B$.

\subsection{Directed lift and triangular kernels}
\label{subsec:common-iteration}
\label{subsec:triangular-kernel}

Define the lifted operator
\begin{equation}
\label{eq:lifted-operator}
M_{H,A,B}\binom{y}{x}:=\binom{Ay+Hx}{Bx-y}.
\end{equation}
The condition $\binom{b}{0}\in M_{H,A,B}\binom{y}{x}$ is exactly \eqref{eq:master-lifted}.

For $\alpha,\beta>0$ and $\theta\in\R$, introduce the kernel
\begin{equation}
\label{eq:triangular-kernel}
K_\theta^{\alpha,\beta}\binom{y}{x}
:=\binom{\alpha^{-1}y-Hx}{-\theta y+\beta^{-1}x}.
\end{equation}
Summing with the lifted operator gives
\begin{equation}
\label{eq:triangular-cancellation}
K_\theta^{\alpha,\beta}+M_{H,A,B}
=
\begin{bmatrix}
\alpha^{-1}I+A&0\\
-(1+\theta)I&\beta^{-1}I+B
\end{bmatrix}.
\end{equation}
The upper-right block vanishes and the individual blocks depend only on the individual operators $A$ and $B$, so the implicit equation \eqref{eq:master-lifted} can be solved using only forward substitution and individual resolvent evaluations.  Applying the proximal point algorithm to \eqref{eq:triangular-cancellation} gives Algorithm \ref{alg:triangular}.

\begin{algorithm}[H]
\caption{One-parameter triangular warped-resolvent iteration}
\label{alg:triangular}
\begin{algorithmic}[1]
\REQUIRE $\alpha,\beta>0$, $\theta\in\R$, and $(y^0,x^0)\in\cX^2$
\FOR{$k=0,1,2,\ldots$}
\STATE $y^{k+1}\leftarrow J_{\alpha A}(y^k-\alpha Hx^k+\alpha b)$
\STATE $\bar y^k\leftarrow(1+\theta)y^{k+1}-\theta y^k$
\STATE $x^{k+1}\leftarrow J_{\beta B}(x^k+\beta\bar y^k)$
\ENDFOR
\end{algorithmic}
\end{algorithm}

\begin{remark}[Well-posedness of Algorithm~\ref{alg:triangular}.]
\label[remark]{rem:step-wellposed}
Maximal monotonicity of $A$ and $B$ and the full domain of $H$ make every line of \Cref{alg:triangular} single-valued. Thus \Cref{alg:triangular} defines a map $T_{\theta,b}:\cX^2\to\cX^2$.   
\end{remark}
The following proposition shows that the fixed points of the map $T_{\theta,b}$ defined by Algorithm~\ref{alg:triangular} are exactly the solutions of the lifted system.

\begin{proposition}[Kernel-independent fixed points]
\label[proposition]{prop:master-fixed-points}
For every $\alpha,\beta>0$ and $\theta\in\R$,
\begin{equation}
\label{eq:master-fixed-points}
\Fix T_{\theta,b}
=
\{(y,x):b\in Ay+Hx,\ y\in Bx\}.
\end{equation}
Hence the second component of every fixed point solves \eqref{eq:master-problem}, and every solution of \eqref{eq:master-problem} admits a lifted fixed point.
\end{proposition}

\begin{proof}
At a fixed point, $\bar y=(1+\theta)y-\theta y=y$. Hence $y=J_{\alpha A}(y-\alpha Hx+\alpha b)$ and $x=J_{\beta B}(x+\beta y)$, equivalently $b-Hx\in Ay$ and $y\in Bx$. Reversing the equivalences proves the converse.
\end{proof}

We next express the iteration as a warped resolvent of a translated lifted operator. The translation incorporates the forcing term $b$ already present in Proposition~\ref{prop:master-fixed-points}.

\begin{definition}[Offset operator]
\label[definition]{def:shifted-warped-resolvent}
For a set-valued operator $M$, a single-valued kernel $K$, and a fixed offset $c$, set $\widehat M_c:=M-c$ and define
\begin{equation}
\label{eq:shifted-warped-resolvent}
T_{K,M,c}(z):=((K+\widehat M_c)^{-1}K)(z)=(K+M)^{-1}(K(z)+c),
\end{equation}
whenever this map is single-valued. Thus $c$ is absorbed into the operator by range translation, and \eqref{eq:shifted-warped-resolvent} is the standard warped resolvent of $\widehat M_c$.
\end{definition}

For the lifted system, the choice $c=(b,0)^\top$ incorporates the existing forcing term $b$ in the first component and preserves the constraint $y\in Bx$ in the second component.

\begin{proposition}[Triangular warped-resolvent identity]
\label[proposition]{prop:triangular-warped-identity}
With $c=(b,0)\T$,
\begin{equation}
\label{eq:triangular-warped-identity}
T_{\theta,b}=T_{K_\theta^{\alpha,\beta},M_{H,A,B},c}
=\bigl(K_\theta^{\alpha,\beta}+M_{H,A,B}-c\bigr)^{-1}K_\theta^{\alpha,\beta}.
\end{equation}
Moreover,
\begin{equation}
\label{eq:warped-fixed-point-identity}
z=T_{K,M,c}z
\quad\Longleftrightarrow\quad
0\in(M-c)z
\quad\Longleftrightarrow\quad
c\in Mz.
\end{equation}
This equivalence recovers the fixed-point characterization in Proposition~\ref{prop:master-fixed-points}.
\end{proposition}

\begin{proof}
Writing $z^+=(y^+,x^+)\T$, the relation $K_\theta^{\alpha,\beta}z+c\in(K_\theta^{\alpha,\beta}+M_{H,A,B})z^+$ and \eqref{eq:triangular-cancellation} give the two rows of \Cref{alg:triangular}. The fixed-point equivalences follow from $Kz\in Kz+Mz-c$.
\end{proof}

The translated lift $M_{H,A,B}-c$ determines the solution set, while the kernel determines the forward-substitution factorization. Thus Algorithm~\ref{alg:triangular} is a warped-resolvent realization of the range-translated lift with a composition-adapted kernel whose sum with the lift is lower triangular.

\subsection{Lift monotonicity and the limits of monotone geometry}
\label{subsec:structural-separation}

In this section, we study conditions under which 
the lifted operator $M_{H,A,B}$ in \eqref{eq:lifted-operator} is monotone on the product space $\cX^2$. We first identify a regime in which it admits a monotone interpretation in a weighted Hilbert product space and the triangular kernel induces a firmly nonexpansive map.  We then characterize cases where the lifted operator is never monotone. \Cref{sec:gain-analysis} complements this geometric viewpoint by proving convergence in these non-monotone cases. We use the following matrix facts in both analyses.

\begin{definition}[$Z$-matrices and $M$-matrices]
 A real square matrix is a $Z$-matrix if its off-diagonal entries are nonpositive. It is an $M$-matrix if it has a representation $sI-L$ with $L\geq0$ entrywise and $s\geq\rho(L)$; it is nonsingular when $s>\rho(L)$.   
\end{definition}
\begin{lemma}
\label[lemma]{lem:z-m-matrix}
 The following properties hold.
\begin{enumerate}[label=\textup{(\roman*)}]
\item For $L\in\R^{n\times n}$ with $L\geq0$ and $\gamma>0$, the conditions $\rho(L)<\gamma$, $\gamma I-L$ being a nonsingular $M$-matrix, $(\gamma I-L)^{-1}$ existing and being entrywise nonnegative, and the existence of $w>0$ with $Lw<\gamma w$ are equivalent. In this case,
\begin{equation}
\label{eq:m-matrix-positive-weight}
(\gamma I-L)^{-1}=\sum_{k=0}^{\infty}\gamma^{-k-1}L^k,
\qquad
w_\gamma:=(\gamma I-L)^{-1}\one>0,
\qquad
Lw_\gamma=\gamma w_\gamma-\one.
\end{equation}
\item For a $2\times2$ $Z$-matrix $C=\begin{bmatrix}a&-b\\-c&d\end{bmatrix}$ with $b,c\geq0$, $C$ is a nonsingular $M$-matrix if all three nonempty principal minors are positive:
\begin{equation}
\label{eq:m-matrix-two-by-two}
a>0,\qquad d>0,\qquad \det C=ad-bc>0.
\end{equation}
Equivalently, it suffices that $a>0$ and $\det C>0$. If $C$ is symmetric, these conditions are also equivalent to $C\succ0$.
\end{enumerate}
\end{lemma}

\begin{proof}
These are standard $M$-matrix characterizations \cite[Chapter~6]{BermanPlemmons1994}. For (i), $\rho(L)<\gamma$ gives the Neumann series in \eqref{eq:m-matrix-positive-weight}; its zeroth term makes $w_\gamma$ strictly positive. Conversely, if the inverse exists and is nonnegative, every row has a positive entry, so $w_\gamma>0$ and $Lw_\gamma=\gamma w_\gamma-\one<\gamma w_\gamma$. More generally, any such $w$ gives, with $D_w=\diag(w_1,\ldots,w_n)$,
\[
\rho(L)=\rho(D_w^{-1}LD_w)\leq\norm{D_w^{-1}LD_w}_\infty=\max_i\frac{(Lw)_i}{w_i}<\gamma.
\]
For (ii), choose $s>\max\{a,d,0\}$ and write $C=sI-L$ with $L\geq0$. If $C$ is a nonsingular $M$-matrix, (i) implies $a,d>0$, and the eigenvalues of $C$ have positive real parts, so $\det C>0$. Conversely, \eqref{eq:m-matrix-two-by-two} gives
\[
C^{-1}=\frac1{ad-bc}\begin{bmatrix}d&b\\c&a\end{bmatrix}\geq0,
\]
and (i) applies. Also, $a>0$ and $ad-bc>0$ imply $d>bc/a\geq0$. The symmetric case follows from the positive-definiteness criterion for a $2\times2$ symmetric matrix.
\end{proof}

For $t>0$, define the bounded, self-adjoint, strongly positive product metric and its induced inner product by
\begin{equation}
\label{eq:product-metric-definition}
\begin{aligned}
P_t&:=\begin{bmatrix}I&0\\0&tI\end{bmatrix}=\diag(I,tI),\\
\inner{z}{\widetilde z}_{P_t}&:=\inner{z}{P_t\widetilde z}=\inner{y}{\widetilde y}+t\inner{x}{\widetilde x},
\qquad
\norm{z}_{P_t}^2=\norm{y}^2+t\norm{x}^2,
\end{aligned}
\end{equation}
where $z=(y,x)\T$, $\widetilde z=(\widetilde y,\widetilde x)\T$, and the unweighted inner product is the usual product inner product on $\cX^2$. These norms are equivalent since $\min\{1,t\}\norm{z}^2\leq\norm{z}_{P_t}^2\leq\max\{1,t\}\norm{z}^2$. For a single-valued map $F:\cX\to\cX$, write $\Lip(F):=\sup_{u\ne v}\norm{Fu-Fv}/\norm{u-v}$.

\begin{proposition}[Monotonicity of the lifted operator]
\label[proposition]{prop:hidden-metric-monotonicity}
Suppose that $A$ and $B$ are maximal monotone and respectively $\mu_A$- and $\mu_B$-strongly monotone, where $\mu_A,\mu_B\geq0$, and that $H:\cX\to\cX$ is Lipschitz continuous. If there exists $t>0$ satisfying
\begin{equation}
\label{eq:effective-threshold}
\Lip(H-tI)^2\leq4t\mu_A\mu_B,
\end{equation}
then $M_{H,A,B}$ is maximally monotone in the Hilbert space $(\cX^2,\inner{\cdot}{\cdot}_{P_t})$. If the inequality is strict, it is strongly monotone in this space. In particular, $H=tI$ makes the lift maximally monotone in this metric for arbitrary maximal monotone $A$ and $B$.
\end{proposition}

\begin{proof}
Choose $a\in Ay$, $\widetilde a\in A\widetilde y$, $d\in Bx$, and $\widetilde d\in B\widetilde x$. Set $\Delta y=y-\widetilde y$, $\Delta x=x-\widetilde x$, $\Delta a=a-\widetilde a$, $\Delta d=d-\widetilde d$, $\Delta H=Hx-H\widetilde x$, and $\delta_t=\Lip(H-tI)$. With $\Delta z=(\Delta y,\Delta x)\T$ and $\Delta M=(\Delta a+\Delta H,\Delta d-\Delta y)\T$, we have
\[
\begin{aligned}
\inner{\Delta z}{\Delta M}_{P_t}
&=\inner{\Delta y}{\Delta a}+t\inner{\Delta x}{\Delta d}+\inner{\Delta y}{\Delta H-t\Delta x}\\
&\geq\mu_A\norm{\Delta y}^2-\delta_t\norm{\Delta y}\norm{\Delta x}+t\mu_B\norm{\Delta x}^2.
\end{aligned}
\]
The last expression is the quadratic form of the symmetric $Z$-matrix
\[
C_t:=\begin{bmatrix}\mu_A&-\delta_t/2\\-\delta_t/2&t\mu_B\end{bmatrix},
\qquad
\det C_t=t\mu_A\mu_B-\delta_t^2/4\geq0,
\]
evaluated at $(\norm{\Delta y},\norm{\Delta x})\T$. Its diagonal entries are nonnegative, so $C_t\succeq0$. Under strict inequality, \Cref{lem:z-m-matrix} (ii) gives $C_t\succ0$, and norm equivalence yields a constant $\mu_t>0$ such that $\inner{\Delta z}{\Delta M}_{P_t}\geq\mu_t\norm{\Delta z}_{P_t}^2$.

For maximality, define $\mathcal D:=\diag(A-\mu_A I,B-\mu_B I)$ and $\mathcal G_t\binom{y}{x}:=\binom{\mu_A y+Hx}{\mu_B x-y}$, so that $M_{H,A,B}=\mathcal D+\mathcal G_t$. Since $A$ and $B$ are maximally monotone and respectively $\mu_A$- and $\mu_B$-strongly monotone, $A-\mu_A I$ and $B-\mu_B I$ are maximally monotone, and hence $\mathcal D$ is maximally monotone in the $P_t$-Hilbert space. The preceding estimate shows that $\mathcal G_t$ is $P_t$-monotone under \eqref{eq:effective-threshold}; moreover, it is single-valued, Lipschitz continuous, and everywhere defined. Hence $\mathcal G_t$ is maximally monotone by \cite[Corollary~20.28]{BauschkeCombettes2017}. Since $\dom\mathcal G_t=\cX^2$, \cite[Corollary~25.5]{BauschkeCombettes2017} implies that $M_{H,A,B}=\mathcal D+\mathcal G_t$ is maximally monotone in the $P_t$-Hilbert space. Translation by $(b,0)\T$ preserves maximal monotonicity.
\end{proof}

The resolvent of a $P_t$-maximally monotone lift is firmly nonexpansive in the metric induced by $P_t$ \cite[Proposition~23.8(iii)]{BauschkeCombettes2017}. To obtain the corresponding interpretation of the triangular map, its kernel must also enter the metric. Indeed, for an invertible linear kernel $K$, the warped-resolvent equation gives
\begin{equation}
\label{eq:preconditioned-resolvent-identity}
p=(K+\widehat M)^{-1}Kz
\ \Longleftrightarrow\ K(z-p)\in\widehat Mp
\ \Longleftrightarrow\ z-p\in Np,
\qquad N:=K^{-1}\widehat M.
\end{equation}
Thus the triangular map is $J_N$, and the required metric should transfer the monotonicity of $\widehat M$ to $N$. The identity $\inner{u}{K^{-1}v}_W=\inner{u}{v}_{P_t}$ suggests $W=P_tK$. This is formalized in the following proposition.

\begin{proposition}[Firm nonexpansiveness in the monotone regime]
\label[proposition]{rem:monotone-lift-averagedness}
Let $A,B:\cX\rightrightarrows\cX$ be maximally monotone, let $b\in\cX$, and let $t,\alpha,\beta>0$ satisfy $\alpha\beta t<1$. Suppose that $H=tI$ and $\theta=1$. Set $K=K_1^{\alpha,\beta}$ and $\widehat M=M_{tI,A,B}-(b,0)\T$, and define
\begin{equation}
\label{eq:compatible-kernel-metric}
W:=P_tK=\begin{bmatrix}\alpha^{-1}I&-tI\\-tI&t\beta^{-1}I\end{bmatrix}.
\end{equation}
Then $W$ is self-adjoint and strongly positive, and $K$ is invertible with bounded inverse. Moreover, $N:=K^{-1}\widehat M$ is maximally monotone in the $W$-metric, and $T_{1,b}=J_N$ is firmly nonexpansive, and hence $1/2$-averaged, in this metric.
\end{proposition}

\begin{proof}
By \Cref{prop:hidden-metric-monotonicity}, $\widehat M$ is maximally monotone in the $P_t$-metric. The choice $\theta=1$ makes the two off-diagonal blocks of $W=P_tK$ equal to $-tI$. Introduce the scalar matrix
\[
C_K:=\begin{bmatrix}\alpha^{-1}&-t\\-t&t\beta^{-1}\end{bmatrix}.
\]
The scalar symmetric $Z$-matrix $C_K$ has positive diagonal entries and $\det C_K=t(\alpha\beta)^{-1}(1-\alpha\beta t)$. Therefore, if $\alpha\beta t<1$, \Cref{lem:z-m-matrix}(ii) gives $C_K\succ0$, and $W$ is self-adjoint and strongly positive. Equivalently,
\[
\inner{z}{Wz}=\alpha^{-1}\norm{y-\alpha t x}^2+t\beta^{-1}(1-\alpha\beta t)\norm{x}^2.
\]
Consequently, $\inner{z}{\widetilde z}_W:=\inner{z}{W\widetilde z}$ defines an equivalent Hilbert inner product, and $K=P_t^{-1}W$ is invertible with bounded inverse.

For $(p_i,v_i)\in\gra N$, $i=1,2$, we have $Kv_i\in\widehat Mp_i$, and hence
\begin{equation}
\label{eq:preconditioned-monotonicity-pairing}
\inner{p_1-p_2}{v_1-v_2}_W
=\inner{p_1-p_2}{K(v_1-v_2)}_{P_t}\geq0.
\end{equation}
This proves $W$-monotonicity of $N$. It also transfers maximality: if $(p,v)$ is monotonically related to $\gra N$ in the $W$-metric, the same identity makes $(p,Kv)$ monotonically related to $\gra\widehat M$ in the $P_t$-metric. Maximality of $\widehat M$ gives $Kv\in\widehat Mp$, and therefore $v\in Np$.

It follows from \eqref{eq:preconditioned-resolvent-identity} and \cite[Proposition~23.8(iii)]{BauschkeCombettes2017} that $T_{1,b}=J_N$ is firmly nonexpansive in the $W$-metric. Explicitly, setting $p=T_{1,b}z$ and $\widetilde p=T_{1,b}\widetilde z$ gives $z-p\in Np$ and $\widetilde z-\widetilde p\in N\widetilde p$, so \eqref{eq:preconditioned-monotonicity-pairing} yields
\begin{equation}
\label{eq:triangular-firm-nonexpansiveness}
\norm{T_{1,b}z-T_{1,b}\widetilde z}_W^2
\leq\inner{T_{1,b}z-T_{1,b}\widetilde z}{z-\widetilde z}_W.
\end{equation}
Equivalently, $2T_{1,b}-I$ is nonexpansive in this metric \cite[Proposition~4.4]{BauschkeCombettes2017}, so $T_{1,b}=\tfrac12I+\tfrac12(2T_{1,b}-I)$ is $1/2$-averaged. The metric $P_tK$ thus supplies the resolvent geometry of the triangular realization.
\end{proof}

Monotonicity of the lift is not guaranteed.  The following proposition introduces a family of problems where skew factors make their resolvents contractive while the directed coupling leaves an eigenvalue with negative real part in the lift. Consequently, the lift is nonmonotone in every equivalent Hilbert inner product, but the triangular iteration still converges.

\begin{proposition}
\label[proposition]{prop:rotational-separation}
Let $m\geq1$, $\nu>0$, $R=\begin{bmatrix}0&-1\\1&0\end{bmatrix}$, $R_m=I_m\otimes R$, and, on $\cX=\R^{2m}$, set
\begin{equation}
\label{eq:rotational-separation-data}
A=B=\nu R_m,\qquad H=-I,\qquad b=0,\qquad \alpha=\beta=1,\qquad \theta=0.
\end{equation}
Then $A$ and $B$ are maximal monotone and
\begin{equation}
\label{eq:rotational-resolvent-factor}
\Lip(J_A)=\Lip(J_B)=\ell_\nu:=(1+\nu^2)^{-1/2}.
\end{equation}
For $\nu>\sqrt3$, the triangular map is a contraction in an equivalent weighted block maximum norm, whereas the lifted operator
\begin{equation}
\label{eq:rotational-lift}
M_\nu=\begin{bmatrix}\nu R_m&-I\\-I&\nu R_m\end{bmatrix}
\end{equation}
is nonmonotone in every equivalent Hilbert inner product. The original composite operator is $H+A\circ B=-(1+\nu^2)I_{2m}$, a negative scalar multiple of the identity, and is therefore nonmonotone in every equivalent Hilbert inner product.
\end{proposition}

\begin{proof}
Since $R_m\T=-R_m$ and $R_m^2=-I$, we have $\inner{h}{\nu R_mh}=0$ for every $h$. Thus $A$ and $B$ are monotone; they are also continuous with full domain, so they are maximal monotone \cite[Corollary~20.28]{BauschkeCombettes2017}. Moreover, the common resolvent has the explicit form
\begin{equation}
\label{eq:rotational-resolvent-explicit}
Q_\nu:=(I+\nu R_m)^{-1}=\frac{I-\nu R_m}{1+\nu^2},
\qquad
Q_\nu\T Q_\nu=\frac1{1+\nu^2}I,
\end{equation}
which proves \eqref{eq:rotational-resolvent-factor}.

For these data, the two updates are $y^+=Q_\nu(y+x)$ and $x^+=Q_\nu(x+y^+)$. Therefore, for the differences of any two inputs,
\[
\norm{\Delta y^+}\leq\ell_\nu(\norm{\Delta y}+\norm{\Delta x}),
\qquad
\norm{\Delta x^+}\leq\ell_\nu\norm{\Delta x}+\ell_\nu\norm{\Delta y^+},
\]
and substitution gives
\begin{equation}
\label{eq:rotational-comparison-matrix}
\binom{\norm{\Delta y^+}}{\norm{\Delta x^+}}
\leq L_\nu\binom{\norm{\Delta y}}{\norm{\Delta x}},
\qquad
L_\nu:=\begin{bmatrix}\ell_\nu&\ell_\nu\\\ell_\nu^2&\ell_\nu+\ell_\nu^2\end{bmatrix}.
\end{equation}
In particular,
\[
I-L_\nu=\begin{bmatrix}1-\ell_\nu&-\ell_\nu\\-\ell_\nu^2&1-\ell_\nu-\ell_\nu^2\end{bmatrix},
\qquad
\det(I-L_\nu)=(1-\ell_\nu)(1-\ell_\nu-\ell_\nu^2)-\ell_\nu^3=1-2\ell_\nu.
\]
If $\nu>\sqrt3$, then $0<\ell_\nu<1/2$, so $1-\ell_\nu>0$, $1-\ell_\nu-\ell_\nu^2>1/4$, and $\det(I-L_\nu)>0$. All three principal minors are positive, and \Cref{lem:z-m-matrix} gives $\rho(L_\nu)<1$. Conversely, $\rho(L_\nu)<1$ implies $\det(I-L_\nu)>0$, so this comparison-matrix condition is equivalent to $\nu>\sqrt3$.

Choose $\gamma\in(\rho(L_\nu),1)$ and $w=(\gamma I-L_\nu)^{-1}\one$. By \Cref{lem:z-m-matrix}(i), $w>0$ and $L_\nu w<\gamma w$. Define $\norm{(y,x)\T}_{w,\infty}:=\max\{\norm{y}/w_1,\norm{x}/w_2\}$. If $\delta=\norm{z-\widetilde z}_{w,\infty}$, then the input-size vector is bounded by $w\delta$, and \eqref{eq:rotational-comparison-matrix} bounds the output-size vector by $L_\nu w\delta\leq\gamma w\delta$. Hence $\norm{T_{0,0}z-T_{0,0}\widetilde z}_{w,\infty}\leq\gamma\norm{z-\widetilde z}_{w,\infty}$.

To test monotonicity of the lift in an arbitrary equivalent Hilbert metric, let $P=P\T\succ0$. A real linear map $M_\nu$ is $P$-monotone exactly when $PM_\nu+M_\nu\T P\succeq0$. The block transformation
\begin{equation}
\label{eq:rotational-lift-diagonalization}
U:=\frac1{\sqrt2}\begin{bmatrix}I&I\\I&-I\end{bmatrix},
\qquad
U\T M_\nu U=\diag(\nu R_m-I,\nu R_m+I)
\end{equation}
shows that the eigenvalues are $-1\pm\mathrm i\nu$ and $1\pm\mathrm i\nu$. For a complex eigenvector $v\ne0$ with $M_\nu v=(-1+\mathrm i\nu)v$, where $v^*$ denotes conjugate transpose,
\begin{equation}
\label{eq:rotational-metric-obstruction}
v^*(PM_\nu+M_\nu\T P)v
=2\RePart(-1+\mathrm i\nu)\,v^*Pv=-2v^*Pv<0.
\end{equation}
This contradicts positive semidefiniteness for every $P\succ0$. Finally, $R_m^2=-I$ gives $H+A\circ B=-I+\nu^2R_m^2=-(1+\nu^2)I$, and for every $G=G\T\succ0$ and $h\ne0$, $\inner{h}{(H+A\circ B)h}_G=-(1+\nu^2)\norm{h}_G^2<0$.
\end{proof}

The monotone regime therefore admits a standard averaged-resolvent interpretation, whereas the rotational family shows that convergence can still be guaranteed in the absence of monotonicity. This separation motivates the direct gain analysis of the explicit triangular map developed in \Cref{sec:gain-analysis}.

\subsection{Proximal realization of coupled convex equilibria}
\label{subsec:proximal-response}

This subsection shows how the triangular splitting iteration for coupled convex equilibria can be implemented through two regularized convex minimization problems, and explains how the parameter $\theta$ connects Jacobi, Gauss--Seidel, and extrapolated updates.

For the coupled convex optimisation equilibrium model of \Cref{sec:problem-classes}, set
\begin{equation}
\label{eq:subdiff-dictionary}
A=\partial g,
\qquad
B=\partial f.
\end{equation}
Then \Cref{alg:triangular} is implemented by two regularized convex optimisation subproblems.

\begin{proposition}[Triangular proximal updates]
\label[proposition]{prop:proximal-response-game}
Under \eqref{eq:subdiff-dictionary}, one iteration is
\begin{subequations}
\label{eq:proximal-response-dynamics}
\begin{align}
y^{k+1}
&=\argmin_v\left\{\mathcal J_g(v;x^k)+\frac{\norm{v-y^k}^2}{2\alpha}\right\},
\label{eq:proximal-response-y}\\
x^{k+1}
&=\argmin_u\left\{\mathcal J_f(u;\bar y^k)+\frac{\norm{u-x^k}^2}{2\beta}\right\},
\qquad
\bar y^k=(1+\theta)y^{k+1}-\theta y^k.
\label{eq:proximal-response-x}
\end{align}
\end{subequations}
Thus $\theta=-1$ gives a proximal Jacobi update, $\theta=0$ gives a proximal Gauss--Seidel update, and $\theta>0$ extrapolates the newly computed first-block update.
\end{proposition}

\begin{proof}
Use $J_{\alpha\partial g}=\prox_{\alpha g}$ and $J_{\beta\partial f}=\prox_{\beta f}$ in \Cref{alg:triangular} and apply the proximal minimization formula.
\end{proof}

\subsection{Specializations of Algorithm~\ref{alg:triangular}}
\label{subsec:classical-specializations}

Algorithm \ref{alg:triangular} specializes to a number of existing algorithms with particular problem data and parameter choices. These specializations are summarized in \Cref{fig:algorithm-family}. 

\emph{Forward--backward branch.} Setting $A=I$, $\alpha=1$, and $\theta=1$ yields the forward--backward iteration after elimination of the intermediate variable, with the proximal-point, gradient, Krasnosel'skii--Mann, and Richardson iterations obtained under the corresponding choices of $B$ and $H$ \cite{Mann1953,BauschkeCombettes2017,RyuYin2022}. 

\emph{Momentum branch.} The choice $A=I$, $B=0$, and $\theta=0$ gives a two-step momentum representation, including Polyak's heavy-ball and inertial fixed-point forms \cite{Polyak1964,Mainge2008}.

\emph{Identity-coupled primal--dual branch.} With $H=I$ and $B=C^{-1}$, the triangular map reduces to the identity-coupled primal--dual family. The reflected endpoint $\theta=1$ gives the Chambolle--Pock/PDHG form, while the additional scaling $\alpha\beta=1$ yields Douglas--Rachford splitting after the standard linear change of variables \cite{ArrowHurwiczUzawa1958,ChambollePock2011,Condat2013,Vu2013}. Douglas--Rachford splitting on the Fenchel dual and ADMM on the primal are equivalent representations of the Eckstein--Bertsekas construction \cite{EcksteinBertsekas1992,OConnorVandenberghe2020}. Finally, setting $H=b=0$ generates the pure-composition branch developed in the next subsection.

\begin{figure}[t]
\centering
\resizebox{\linewidth}{!}{%
\begin{tikzpicture}[
  font=\footnotesize,
  >=stealth,
  master/.style={draw, thick, rounded corners=3pt, fill=black!4,
                 align=center, inner sep=5pt},
  family/.style={draw, rounded corners=2pt, fill=black!4, align=flush center,
                 text width=3.5cm, minimum height=1.25cm, inner sep=2pt},
  method/.style={draw, rounded corners=2pt, align=flush center,
                 text width=3.25cm, minimum height=1.25cm, inner sep=2pt},
]
\def\csep{3.75}   
\def\famy{-1.65}  
\def\firsty{-3.25} 
\def\rsep{1.4}    
\def\algnode#1#2#3{%
  \node[method] (m#1-#2) at ({(#1-1)*\csep},{\firsty-(#2-1)*\rsep}) {#3};}
\def\algframe#1#2{
  \draw[rounded corners=3pt]
    ([xshift=-4pt,yshift=4pt]m#1-1.north west) rectangle
    ([xshift=4pt,yshift=-4pt]m#1-#2.south east);
  \coordinate (g#1) at ([yshift=4pt]m#1-1.north);}

\node[master] (top) at ({1.5*\csep},0)
  {\textbf{Triangular master map}\\[2pt] $0 \in Hx + A(Bx) - b$};

\node[family] (f1) at ({0*\csep},\famy)
  {$A = I,\ \alpha = 1,\ \theta = 1$\\[2pt]\textbf{Forward--backward}\\\textbf{splitting}};
\node[family] (f2) at ({1*\csep},\famy)
  {$A = I,\ B = 0,\ \theta = 0$\\[2pt]\textbf{Momentum}\\\textbf{representation}};
\node[family] (f3) at ({2*\csep},\famy)
  {$H = I,\ B = C^{-1}$\\[2pt]\textbf{Identity-coupled}\\\textbf{primal--dual family}};
\node[family] (f4) at ({3*\csep},\famy)
  {$H = 0,\ b = 0$\\[2pt]\textbf{Pure-composition}\\\textbf{family}};

\algnode{1}{1}{proximal-point\\algorithm}
\algnode{1}{2}{gradient descent}
\algnode{1}{3}{Krasnosel'ski\u{\i}--Mann\\and Picard iterations}
\algnode{1}{4}{Richardson iteration}

\algnode{2}{1}{Polyak heavy-ball\\method}
\algnode{2}{2}{inertial Krasnosel'ski\u{\i}--\\Mann-type iteration}

\algnode{3}{1}{Arrow--Hurwicz method}
\algnode{3}{2}{Chambolle--Pock /\\primal--dual hybrid\\gradient method}
\algnode{3}{3}{Douglas--Rachford\\splitting\\$\theta = 1,\ \alpha\beta = 1$}
\algnode{3}{4}{alternating direction\\method of multipliers}

\algnode{4}{1}{original composition\\iteration\\$\theta = -1$}
\algnode{4}{2}{sequential composition\\iteration\\$\theta = 0$}
\algnode{4}{3}{reflected composition\\iteration\\$\theta = 1$}
\algnode{4}{4}{subdifferential\\composition models}

\algframe{1}{4}
\algframe{2}{2}
\algframe{3}{4}
\algframe{4}{4}

\foreach \i in {1,...,4} {
  \draw[->] (top.south) -- (f\i.north);
  \draw[->] (f\i.south) -- (g\i);
}
\end{tikzpicture}%
}
\caption{Specializations of Algorithm~\ref{alg:triangular}. The upper branches recover established splitting schemes through exact choices of the problem data and algorithmic parameters, while $H=b=0$ generates the pure-composition family.}
\label{fig:algorithm-family}
\end{figure}
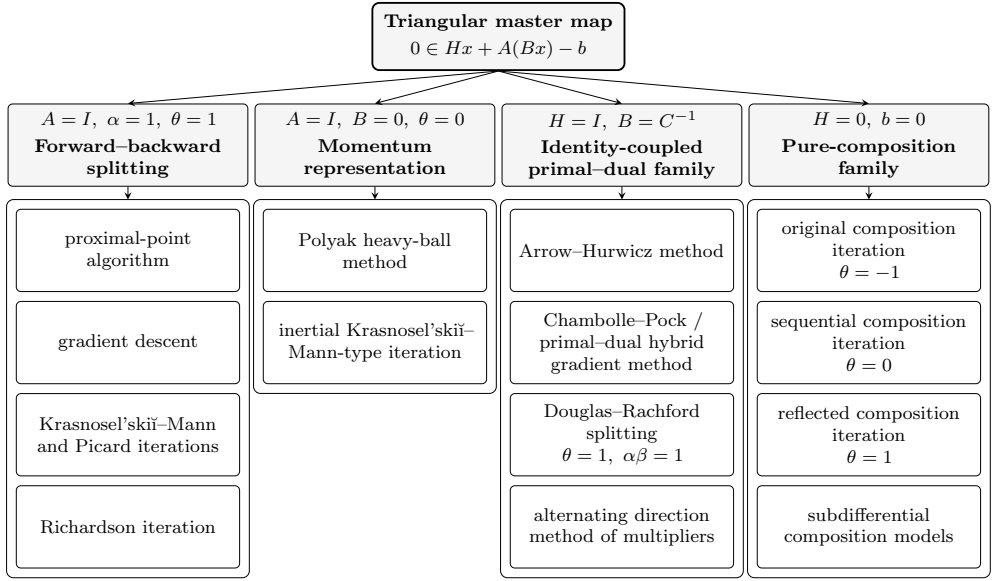

\Cref{fig:algorithm-family} records the algebraic specialization structure of the triangular family. The recovered branches retain their standard convergence theories under their native assumptions, while \Cref{sec:gain-analysis} supplies an additional contraction certificate whenever the corresponding component gains satisfy the Schur condition.

\subsection{New factorwise algorithms for pure compositions}
\label{subsec:pure-kernels}

The preceding specializations recover established algorithms by choosing particular problem data. A distinct branch is obtained by setting $H=b=0$: the method then acts directly on the composition $A\circ B$ without converting it to a monotone sum or linearizing either factor. This yields a factorwise algorithm family whose only implicit operations are $J_{\alpha A}$ and $J_{\beta B}$. The construction is one of the principal new consequences of the triangular kernel.

The parameter $\theta$ selects the intermediate signal entering the second resolvent, as summarized in Table~\ref{tab:kernel-endpoints}.
\begin{table}[H]
\centering
\caption{Signal transmission in the pure-composition branch.}
\label{tab:kernel-endpoints}
\begin{tabular}{@{}c c l@{}}
\toprule
$\theta$ & transmitted signal $\bar y^k$ & interpretation \\
\midrule
$-1$ & $y^k$ & previous state composition update \\
$0$ & $y^{k+1}$ & sequential composition update \\
$1$ & $2y^{k+1}-y^k$ & reflected composition update \\
\bottomrule
\end{tabular}
\end{table}

The corresponding problem is
\begin{equation}
\label{eq:pure-composition}
0\in A(Bx)
\end{equation}
and the factorwise family
\begin{subequations}
\label{eq:pure-family}
\begin{align}
y^{k+1}&=J_{\alpha A}y^k,
\label{eq:pure-family-y}\\
x^{k+1}&=J_{\beta B}\bigl(x^k+\beta((1+\theta)y^{k+1}-\theta y^k)\bigr).
\label{eq:pure-family-x}
\end{align}
\end{subequations}

The cases $\theta=-1$ and $\theta=0$ differ only by an index shift of the variable entering from the first row. Indeed, setting $\widetilde y^k:=y^{k+1}$ transforms the $\theta=0$ iteration into the $\theta=-1$ form with initialization $\widetilde y^0=J_{\alpha A}y^0$, and hence yields the same $x$-sequence under the corresponding initialization. The endpoint $\theta=1$ instead transmits the reflected state $2y^{k+1}-y^k$. All three parameter choices share the same fixed-point set, as shown in the following proposition.

\begin{proposition}[Common fixed points of the composition family]
\label[proposition]{prop:pure-fixed}
Define
\begin{equation}
\label{eq:pure-solution-set}
\mathcal S_{A,B}:=\{(y,x):0\in Ay,\ y\in Bx\}.
\end{equation}
Then for every $\theta\in\R$, the fixed points of \eqref{eq:pure-family} are exactly $\mathcal S_{A,B}$. Hence the second component $x$ solves \eqref{eq:pure-composition}, and every solution admits a lifted fixed point.
\end{proposition}

\begin{proof}
At a fixed point the transmitted signal equals $y$, so the two resolvent identities are equivalent to $0\in Ay$ and $y\in Bx$.
\end{proof}

\section{Gain-matrix convergence and phase-resolved tightness}
\label{sec:certificate-calibration}
\label{sec:gain-analysis}

The kernel construction determines an explicit fixed-point map. Its convergence is governed by the gains propagated through the two resolvent rows. In this section, we develop a comparison matrix argument for proving convergence. To do so, we use the componentwise small-gain viewpoint common in the study of interconnected systems, with Perron scaling converting a nonnegative gain matrix into a contraction certificate \cite{Siljak1978,JiangTeelPraly1994}. We first prove a global convergence theorem, then characterize sharpness of the conditions, and finally calibrate the certificate on nonsmooth proximal instances.

\subsection{Global two-block contraction}
\label{subsec:global-contraction}

For $z=(z_1,z_2)\T\in\cX^2$, define
\begin{equation}
\label{eq:block-size-map}
e(z):=(\norm{z_1},\norm{z_2})\T,
\qquad
\norm{z}_{w,\infty}:=\max_i\frac{\norm{z_i}}{w_i},
\qquad
\norm{z}_{P,2}^2:=p_1\norm{z_1}^2+p_2\norm{z_2}^2,
\end{equation}
where $w\in\Rpp^2$ and $P=\diag(p_1,p_2)\succ0$. It is a standard result (see e.g. \cite{BauschkeCombettes2017}) that, for a $\mu$-strongly monotone maximal operator,
\begin{equation}
\label{eq:resolvent-strong-bound}
\Lip(J_{\alpha A})\leq(1+\alpha\mu)^{-1}.
\end{equation}

\begin{theorem}[Global two-block small-gain contraction]
\label[theorem]{thm:global}
Assume that $H$ is $L_H$-Lipschitz and that
\begin{equation}
\label{eq:component-lipschitz}
\Lip(J_{\alpha A})\leq\ell_A,
\qquad
\Lip(J_{\beta B})\leq\ell_B,
\qquad
\ell_A,\ell_B\in[0,1].
\end{equation}
For $\theta\in[-1,1]$, define
\begin{equation}
\label{eq:master-gain}
L_\theta:=
\begin{bmatrix}
\ell_A&\alpha\ell_A L_H\\
\beta\ell_B((1+\theta)\ell_A+|\theta|)&
\ell_B(1+\alpha\beta(1+\theta)\ell_A L_H)
\end{bmatrix},
\end{equation}
and suppose
\begin{equation}
\label{eq:master-schur}
\rho(L_\theta)<1.
\end{equation}
Then:
\begin{enumerate}
\item for every $\gamma\in(\rho(L_\theta),1)$,
\begin{equation}
\label{eq:perron-contraction}
w_\gamma:=(\gamma I-L_\theta)^{-1}\one\in\Rpp^2,
\qquad
\norm{T_{\theta,b}z-T_{\theta,b}\widetilde z}_{w_\gamma,\infty}
\leq\gamma\norm{z-\widetilde z}_{w_\gamma,\infty};
\end{equation}
\item with
\begin{equation}
\label{eq:hilbert-weight-construction}
r:=(I-L_\theta)^{-1}\one,
\quad
s:=(I-L_\theta\T)^{-1}\one,
\quad
P:=\diag(s_1/r_1,s_2/r_2),
\end{equation}
\begin{equation}
\label{eq:hilbert-rate-construction}
\eta_r:=\max_i\frac{(L_\theta r)_i}{r_i},
\quad
\eta_s:=\max_i\frac{(L_\theta\T s)_i}{s_i},
\quad
q:=\sqrt{\eta_r\eta_s}<1,
\end{equation}
one has
\begin{equation}
\label{eq:hilbert-contraction}
\norm{T_{\theta,b}z-T_{\theta,b}\widetilde z}_{P,2}
\leq q\norm{z-\widetilde z}_{P,2};
\end{equation}
\item $T_{\theta,b}$ has a unique fixed point $z^\star$, and the Picard iterates satisfy
\begin{equation}
\label{eq:linear-convergence}
\norm{z^k-z^\star}_{P,2}\leq q^k\norm{z^0-z^\star}_{P,2},
\qquad
\norm{z^k-z^\star}_{P,2}
\leq\frac{\norm{z^{k+1}-z^k}_{P,2}}{1-q};
\end{equation}
moreover,
\begin{equation}
\label{eq:componentwise-residual-bound}
e(z-z^\star)\leq(I-L_\theta)^{-1}e(z-T_{\theta,b}z).
\end{equation}
\end{enumerate}
\end{theorem}

\begin{proof}
Let $z=(y,x)\T$ and $\widetilde z=(\widetilde y,\widetilde x)\T$. The first row of \Cref{alg:triangular} gives
\begin{equation}
\label{eq:first-row-gain}
\norm{y^+-\widetilde y^+}
\leq\ell_A\norm{y-\widetilde y}+\alpha\ell_A L_H\norm{x-\widetilde x}.
\end{equation}
Since $1+\theta\geq0$, the second row satisfies
\begin{align}
\norm{x^+-\widetilde x^+}
&\leq\ell_B\norm{x-\widetilde x}
 +\beta\ell_B\bigl((1+\theta)\norm{y^+-\widetilde y^+}+|\theta|\norm{y-\widetilde y}\bigr)\notag\\
&\leq\beta\ell_B((1+\theta)\ell_A+|\theta|)\norm{y-\widetilde y}
 +\ell_B(1+\alpha\beta(1+\theta)\ell_A L_H)\norm{x-\widetilde x}.
\label{eq:second-row-gain}
\end{align}
Hence
\begin{equation}
\label{eq:componentwise-master-gain}
e(T_{\theta,b}z-T_{\theta,b}\widetilde z)\leq L_\theta e(z-\widetilde z).
\end{equation}
For $\gamma>\rho(L_\theta)$, \Cref{lem:z-m-matrix} (i) gives $w_\gamma>0$ and $L_\theta w_\gamma=\gamma w_\gamma-\one<\gamma w_\gamma$. If $\delta=\norm{z-\widetilde z}_{w_\gamma,\infty}$, then $e(z-\widetilde z)\leq w_\gamma\delta$, so \eqref{eq:componentwise-master-gain} yields $e(T_{\theta,b}z-T_{\theta,b}\widetilde z)\leq\gamma w_\gamma\delta$. This is \eqref{eq:perron-contraction}.

The same lemma applied to $L_\theta$ and $L_\theta\T$ with $\gamma=1$ gives $r,s>0$, while $L_\theta r=r-\one$ and $L_\theta\T s=s-\one$ imply $0\leq\eta_r,\eta_s<1$. For $d\in\Rplus^2$, weighted Cauchy--Schwarz first gives
\[
\bigl((L_\theta d)_i\bigr)^2
=\left(\sum_j\sqrt{(L_\theta)_{ij}r_j}\sqrt{\frac{(L_\theta)_{ij}}{r_j}}\,d_j\right)^2
\leq(L_\theta r)_i\sum_j(L_\theta)_{ij}\frac{d_j^2}{r_j}.
\]
Multiplication by $s_i/r_i$, summation over $i$, and the bounds $L_\theta r\leq\eta_r r$ and $L_\theta\T s\leq\eta_s s$ therefore yield
\begin{equation}
\label{eq:quadratic-chain}
(L_\theta d)\T P(L_\theta d)
\leq\eta_r\sum_{i,j}s_i(L_\theta)_{ij}\frac{d_j^2}{r_j}
=\eta_r\sum_j(L_\theta\T s)_j\frac{d_j^2}{r_j}
\leq\eta_r\eta_s d\T P d.
\end{equation}
Taking $d=e(z-\widetilde z)$ proves \eqref{eq:hilbert-contraction}. Banach--Picard gives uniqueness and \eqref{eq:linear-convergence}. Finally, the triangle inequality and \eqref{eq:componentwise-master-gain} give $(I-L_\theta)e(z-z^\star)\leq e(z-T_{\theta,b}z)$. Multiplication by the entrywise nonnegative inverse $(I-L_\theta)^{-1}$ from \Cref{lem:z-m-matrix} (i) preserves this inequality and proves \eqref{eq:componentwise-residual-bound}.
\end{proof}

\Cref{app:gain} extends these constructions to arbitrary finite block gains and shows that diagonal rescaling of the block maximum norm achieves the Perron root as its optimal contraction factor. The appendix also gives the corresponding diagonal Hilbert-product construction.

The following proposition gives a characterization of the Schur criterion in Theorem~\ref{thm:global} in terms of the characteristics of the individual operators, by exploiting the fact that $I-L_\theta$ is a nonsingular $M$-matrix.

\begin{proposition}[Exact scalar Schur criterion]
\label[proposition]{prop:schur-criterion}
Let
\begin{equation}
\label{eq:chi-theta}
\chi_\theta:=1+\theta+|\theta|
=
\begin{cases}
1,&-1\leq\theta\leq0,\\
1+2\theta,&0\leq\theta\leq1.
\end{cases}
\end{equation}
Then
\begin{equation}
\label{eq:scalar-schur-condition}
\rho(L_\theta)<1
\quad\Longleftrightarrow\quad
(1-\ell_A)(1-\ell_B)>
\alpha\beta\ell_A\ell_B L_H\chi_\theta.
\end{equation}
\end{proposition}

\begin{proof}
Write $L_\theta=\begin{bmatrix}a&b_0\\c&d\end{bmatrix}\geq0$, where $a=\ell_A$ and $d=\ell_B+\alpha\beta(1+\theta)\ell_A\ell_B L_H$. Then $I-L_\theta=\begin{bmatrix}1-a&-b_0\\-c&1-d\end{bmatrix}$ is a $Z$-matrix. Expanding its determinant and collecting the coupling terms gives
\begin{equation}
\label{eq:master-gain-determinant}
\begin{aligned}
\det(I-L_\theta)
&=(1-\ell_A)\bigl(1-\ell_B-\alpha\beta(1+\theta)\ell_A\ell_B L_H\bigr)\\
&\quad-\alpha\beta\ell_A\ell_B L_H\bigl((1+\theta)\ell_A+|\theta|\bigr)\\
&=(1-\ell_A)(1-\ell_B)-\alpha\beta\ell_A\ell_B L_H\chi_\theta.
\end{aligned}
\end{equation}
Suppose first that the strict inequality in \eqref{eq:scalar-schur-condition} holds. Its right-hand side is nonnegative, so $\ell_A,\ell_B\in[0,1]$ imply $\ell_A<1$ and $\ell_B<1$. Thus $1-a>0$ and \eqref{eq:master-gain-determinant} gives $\det(I-L_\theta)>0$. The second diagonal principal minor is positive as well, because
\[
(1-a)(1-d)=\det(I-L_\theta)+b_0c>0,
\qquad
1-d=\frac{\det(I-L_\theta)+b_0c}{1-a}>0.
\]
All three nonempty principal minors of $I-L_\theta$ are therefore positive. By \Cref{lem:z-m-matrix} (ii), it is a nonsingular $M$-matrix, and part (i) gives $\rho(L_\theta)<1$.

Conversely, if $\rho(L_\theta)<1$, the same lemma makes $I-L_\theta$ a nonsingular $M$-matrix with positive principal minors. In particular, $\det(I-L_\theta)>0$, and \eqref{eq:master-gain-determinant} gives the strict scalar inequality.
\end{proof}

\begin{corollary}[Strong-monotonicity threshold]
\label[corollary]{cor:canonical-threshold}
If $A$ and $B$ are respectively $\mu_A$- and $\mu_B$-strongly monotone and
\begin{equation}
\label{eq:canonical-resolvent-bounds}
\ell_A=(1+\alpha\mu_A)^{-1},
\qquad
\ell_B=(1+\beta\mu_B)^{-1},
\end{equation}
then \eqref{eq:master-schur} is equivalent to
\begin{equation}
\label{eq:canonical-threshold}
\mu_A\mu_B>L_H\chi_\theta.
\end{equation}
\end{corollary}

\begin{proof}
The bounds in \eqref{eq:canonical-resolvent-bounds} satisfy $1-\ell_A=\alpha\mu_A\ell_A$ and $1-\ell_B=\beta\mu_B\ell_B$. Consequently,
\[
(1-\ell_A)(1-\ell_B)-\alpha\beta\ell_A\ell_B L_H\chi_\theta
=\alpha\beta\ell_A\ell_B\bigl(\mu_A\mu_B-L_H\chi_\theta\bigr).
\]
Since $\alpha\beta\ell_A\ell_B>0$, \Cref{prop:schur-criterion} gives exactly \eqref{eq:canonical-threshold}.
\end{proof}

The Schur threshold and the block-metric condition in \Cref{prop:hidden-metric-monotonicity} describe complementary geometric regimes. For $H=tI$, block-metric monotonicity holds for arbitrary maximal monotone $A$ and $B$, whereas \eqref{eq:canonical-threshold} requires $\mu_A\mu_B>t\chi_\theta$. Conversely, \Cref{prop:rotational-separation} satisfies the gain condition while its lift is nonmonotone in every equivalent Hilbert inner product. Block-metric monotonicity allows convergence to be guaranteed by averagedness of the resolvent, while Schur stability guarantees linear convergence via the individual operator gains.

For the potential model \eqref{eq:weighted-potential}, the threshold \eqref{eq:canonical-threshold} has a direct curvature interpretation.

\begin{corollary}[Strong convexity of the bilinear potential]
\label[corollary]{cor:potential-curvature}
Assume that $g$ and $f$ are $\mu_A$- and $\mu_B$-strongly convex and $H=-\gamma I$. If
\begin{equation}
\label{eq:potential-curvature-threshold}
\mu_A\mu_B>\gamma,
\end{equation}
then $\mathcal P$ in \eqref{eq:weighted-potential} is strongly convex and has a unique minimizer. For the quadratic functions $g(y)=\mu_A\|y\|^2/2$ and $f(x)=\mu_B\|x\|^2/2$, this condition is also necessary for strong convexity of $\mathcal P$. In this case, it also coincides with the Schur condition \eqref{eq:canonical-threshold} for $\theta\in[-1,0]$.

\end{corollary}

\begin{proof}
Write $g=g_0+\mu_A\|\cdot\|^2/2$ and $f=f_0+\mu_B\|\cdot\|^2/2$, where $g_0$ and $f_0$ are convex. Then $\mathcal P$ is the sum of $g_0(y)+\gamma f_0(x)-\langle b,y\rangle$ and a quadratic form with block coefficient $C_\gamma\otimes I$, where
\[
C_\gamma=\begin{pmatrix}\mu_A&-\gamma\\-\gamma&\gamma\mu_B\end{pmatrix}.
\]
The condition $\mu_A\mu_B>\gamma$ makes $C_\gamma$ positive definite. Hence $\mathcal P$ is proper, lower semicontinuous, and strongly convex, and attains a unique minimum. For the stated isotropic family, $C_\gamma\otimes I$ is the exact Hessian, which proves necessity as well.
\end{proof}

\subsection{Sharpness of the gain certificate}
\label{subsec:phase-attainability}

In this subsection, we characterize the sharpness of the gain certificate of Theorem~\ref{thm:global}, by manipulating the phase of the operator $H$. Phase theory complements the magnitude-based small-gain analysis used above by retaining angular information \cite{ChenWangKhongQiu2024, RyuHannahYin2022}. Here the scalar data entering $L_\theta$ are held fixed while the eigenvalue phase of $H$ varies along a rotation path.

\begin{remark}[Information retained by the comparison matrix]
\label[remark]{rem:global-hypothesis}
The matrix $L_\theta$ depends only on the transfer gains determined by $(\ell_A,\ell_B,L_H,\alpha,\beta,\theta)$. Any phase information in $L_\theta$ therefore describes the interconnection structure and not the individual operators. For $\theta\in[-1,0]$, the feasibility boundary is constant because $\chi_\theta=1$, although the certified factor varies with $\theta$. For $\theta>0$, the region of guaranteed convergence contracts with the extrapolation level.
\end{remark}

\begin{proposition}[Worst case phase]
\label[proposition]{prop:worst-phase-attainability}
Fix $\alpha,\beta>0$, $L_H\geq0$, $\ell_A,\ell_B\in(0,1]$, and $\theta\in[-1,0]$. On $\cX=\R$, set
\begin{equation}
\label{eq:worst-phase-linear-data}
A=\frac{1-\ell_A}{\alpha\ell_A}I,
\quad
B=\frac{1-\ell_B}{\beta\ell_B}I,
\quad
H=-L_HI,
\quad b=0.
\end{equation}
Then $J_{\alpha A}=\ell_AI$, $J_{\beta B}=\ell_BI$, and
\begin{equation}
\label{eq:worst-phase-exact-recursion}
\binom{y^{k+1}}{x^{k+1}}
=L_\theta\binom{y^k}{x^k}.
\end{equation}
Hence $\rho(L_\theta)<1$ is necessary and sufficient for this instance to be strictly contractive in an equivalent norm, or equivalently for its zero fixed point to be globally asymptotically stable. Consequently, the scalar gain condition of Theorem~\ref{thm:global} is sharp for uniform contraction.
\end{proposition}

\begin{proof}
The first row is $y^{k+1}=\ell_A(y^k+\alpha L_Hx^k)$. Substitution into the second gives the lower-left coefficient $\beta\ell_B((1+\theta)\ell_A-\theta)$, which equals the corresponding entry of $L_\theta$ because $-\theta=|\theta|$.
If $\rho(L_\theta)<1$, the Perron-weighted contraction estimate applies. Conversely, let $p\ne0$ be a nonnegative Perron vector satisfying $L_\theta p=\rho(L_\theta)p$. If $\rho(L_\theta)\ge1$, the orbit from $p$ does not converge to zero, and $\|L_\theta p\|\ge\|p\|$ in every norm. This proves necessity and the asserted uniform sharpness.
\end{proof}

\begin{remark}
\label[remark]{rem:positive-theta-phase-gap}
The analysis above holds only for $\theta \in [-1, 0]$, and whether or not Theorem~\ref{thm:global} is conservative for $\theta \in (0, 1]$ in general remains an open question.  The following construction gives a class of problems where Theorem~\ref{thm:global} is provably conservative.

To vary the coupling phase while retaining the same scalar gains, fix $\alpha,\beta,L_H>0$, $\ell_A,\ell_B\in(0,1]$, and $\theta\in[-1,1]$, and set
\begin{equation}
\label{eq:phase-rotation-family}
R_\varphi=\begin{bmatrix}\cos\varphi&-\sin\varphi\\\sin\varphi&\cos\varphi\end{bmatrix},
\qquad
M_\varphi=I_m\otimes R_\varphi,
\qquad
\varphi\in[0,\pi].
\end{equation}
Take $A=(1-\ell_A)(\alpha\ell_A)^{-1}I$, $B=(1-\ell_B)(\beta\ell_B)^{-1}I$, $H=L_HM_\varphi$, and $b=0$. The two complex rotation modes $e^{\pm\mathrm i\varphi}$ give conjugate iteration matrices with the same spectral radius. On the $e^{\mathrm i\varphi}$-mode, substitution of the first update into the second gives
\begin{equation}
\label{eq:positive-theta-exact-modal-matrix}
T_\theta(\varphi)=\begin{bmatrix}
\ell_A&-\alpha\ell_A L_He^{\mathrm i\varphi}\\
\beta\ell_B((1+\theta)\ell_A-\theta)&\ell_B(1-\alpha\beta(1+\theta)\ell_A L_He^{\mathrm i\varphi})
\end{bmatrix}.
\end{equation}

For this rotation family, every $\theta\in(0,1]$ produces cancellation in the lower-left transfer, since
\begin{equation}
\label{eq:positive-theta-lower-left-gap}
(1+\theta)\ell_A+\theta-|(1+\theta)\ell_A-\theta|
=2\min\{(1+\theta)\ell_A,\theta\}>0.
\end{equation}
Let $\widehat L_\theta$ agree with $L_\theta$ except for
\begin{equation}
\label{eq:positive-theta-intermediate-gain}
(\widehat L_\theta)_{21}:=\beta\ell_B|(1+\theta)\ell_A-\theta|<(L_\theta)_{21}.
\end{equation}
Then $|T_\theta(\varphi)|\leq\widehat L_\theta\leq L_\theta$, with a strict inequality in that entry. Irreducibility of $L_\theta$ and Perron--Frobenius monotonicity \cite[Chapter~2]{BermanPlemmons1994} give
\begin{equation}
\label{eq:positive-theta-strict-gap}
\rho(T_\theta(\varphi))\leq\rho(\widehat L_\theta)<\rho(L_\theta)
\qquad(0\leq\varphi\leq\pi).
\end{equation}
Thus no phase in this family attains the canonical gain when $\theta>0$, whereas $\theta\leq0$ attains it at $\varphi=\pi$. The following numerical scan illustrates both regimes.
\end{remark}

To calibrate both regimes on one fixed-gain path, specialize the rotation family in \eqref{eq:phase-rotation-family} to $m=50$. Then $M_0=I$, $M_{\pi/2}$ is skew-adjoint, $M_\pi=-I$, and $\norm{M_\varphi}=1$, while the comparison matrix \eqref{eq:master-gain} is independent of $\varphi$. Take
\begin{equation}
\label{eq:phase-sharpness-model}
g(y)=\frac{\mu_A}{2}\norm{y}^2,
\qquad
f(x)=\frac{\mu_B}{2}\norm{x}^2,
\qquad
H=M_\varphi,
\qquad
b=0,
\end{equation}
with
\begin{equation}
\label{eq:phase-sharpness-parameters}
\mu_A=\mu_B=2,
\qquad
\alpha=\beta=L_H=1,
\end{equation}
so $\ell_A=\ell_B=1/3$. On the $e^{\mathrm i\varphi}$-mode,
\begin{equation}
\label{eq:phase-exact-recursion}
\binom{\widehat e_y^{k+1}}{\widehat e_x^{k+1}}
=T_\theta(\varphi)
\binom{\widehat e_y^k}{\widehat e_x^k},
\end{equation}
where $T_\theta(\varphi)$ is the modal iteration matrix in \eqref{eq:positive-theta-exact-modal-matrix}. Define
\begin{equation}
\label{eq:phase-sharpness-ratio}
r_\theta(\varphi):=\rho(T_\theta(\varphi)),
\qquad
\kappa_\theta(\varphi):=\rho(L_\theta)/r_\theta(\varphi).
\end{equation}
The canonical roots are
\begin{equation}
\label{eq:phase-gain-roots}
\rho(L_{-1})\approx0.667,
\quad
\rho(L_0)\approx0.589,
\quad
\rho(L_{1/2})\approx0.760,
\quad
\rho(L_1)\approx0.889.
\end{equation}

For $\theta=-1$ and $0$, the ratio decreases to one at $\varphi=\pi$, where the exact matrix equals the canonical gain, in agreement with \Cref{prop:worst-phase-attainability}. Positive extrapolation remains strictly below the canonical root for every phase, as predicted by  the construction in \Cref{rem:positive-theta-phase-gap}. Numerically, $\kappa_{1/2}$ decreases from $2.281$ to $1.521$ and $\kappa_1$ from $2.000$ to $1.886$; at $\theta=1/2$, $(1+\theta)\ell_A-\theta=0$ removes the inherited-state transfer.

\begin{figure}[t]
\centering
\includegraphics[width=0.80\textwidth]{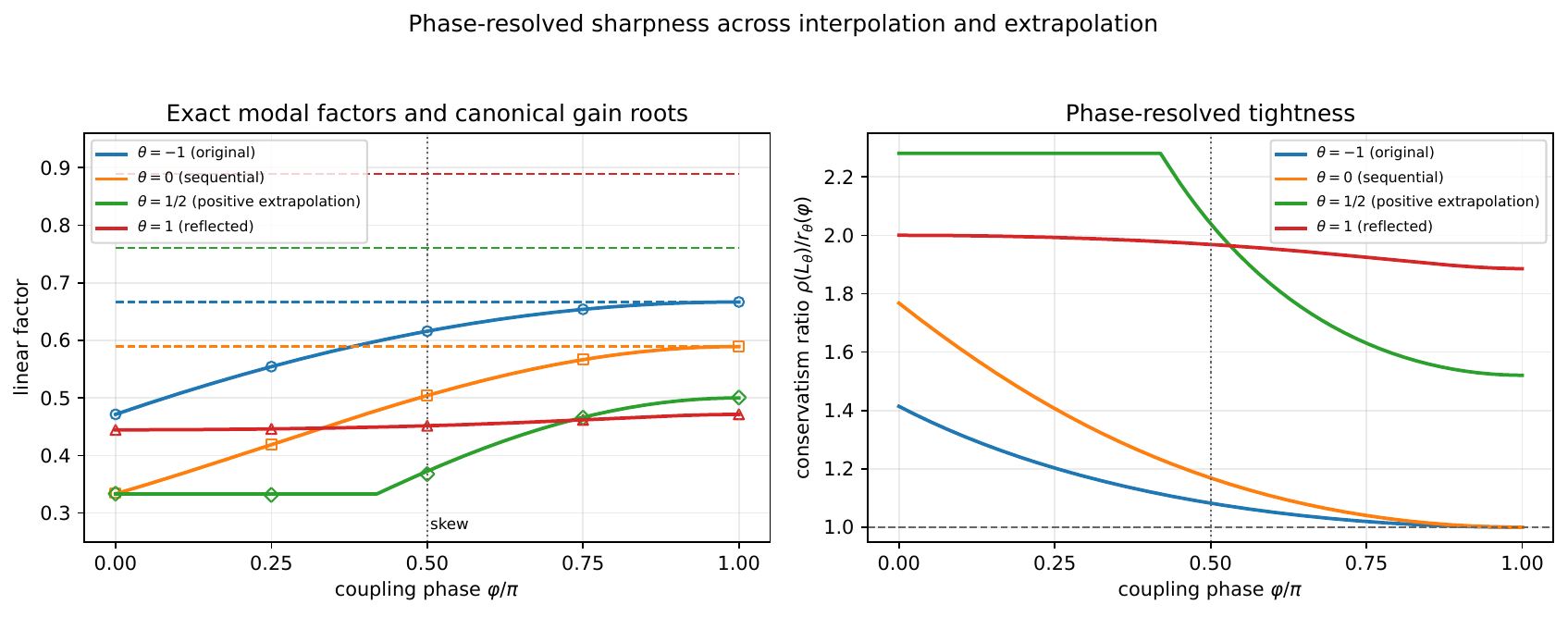}
\caption{Exact modal factors and certificate-to-exact ratios along the fixed-gain phase path. Nonpositive $\theta$ attains the canonical gain at $M_\pi=-I$, while positive extrapolation retains the strict gap predicted by \Cref{rem:positive-theta-phase-gap}.}
\label{fig:phase-resolved-sharpness}
\end{figure}

Together, the analytical families and numerical scans of this section establish a reliable uniform certificate, exact worst-phase attainment in the nonpositive regime, and measurable phase-dependent conservatism away from that regime.

\subsection{Nonsmooth certificate and rate calibration}
\label{subsec:certificate-calibration}

The preceding analysis gives a calibration of the convergence conditions of Theorem~\ref{thm:global} for linear instances. We next test the same magnitude-only certificate on nonsmooth proximal maps. The two couplings below have the same operator norm and therefore share the same convergence boundary as the linear problems, while their observed convergence regions and tail factors reveal the influence of signed and spectral structure.

Let $n\ge4$ be even, and let $S$ be the cyclic right shift on $\mathbb R^n$. Define
\begin{equation}
\label{eq:certificate-couplings}
M_{\rm sym}:=\frac{2I-S-S^*}{4},
\qquad
M_{\rm ns}:=S,
\qquad
\norm{M_{\rm sym}}=\norm{M_{\rm ns}}=1.
\end{equation}
The two couplings therefore generate the same canonical magnitude-based certificate, while retaining different structure: $M_{\rm sym}$ is self-adjoint and $M_{\rm ns}$ is a nonsymmetric cyclic shift.

With $H=M$, $b=0$, $\alpha=\beta=1$, and $\theta=0$, consider
\begin{equation}
\label{eq:certificate-functions}
g_{\mu_A}(y)=\frac{\mu_A}{2}\norm{y-c}^2+0.12\norm{y}_1,
\qquad
f_{\mu_B}(x)=\frac{\mu_B}{2}\norm{x-d}^2+0.10\norm{x}_1.
\end{equation}
Choose $(y^\star,x^\star)$ with nonzero coordinates and define
\begin{equation}
\label{eq:certificate-exact-centers}
c=y^\star+\frac{Mx^\star+0.12\operatorname{sgn}(y^\star)}{\mu_A},
\qquad
d=x^\star-\frac{y^\star-0.10\operatorname{sgn}(x^\star)}{\mu_B}.
\end{equation}
Then $-Mx^\star\in\partial g_{\mu_A}(y^\star)$ and $y^\star\in\partial f_{\mu_B}(x^\star)$, so the reference pair is an exact lifted solution. Since $L_H=\chi_0=1$, the common convergence boundary is
\begin{equation}
\label{eq:certificate-boundary}
\mu_A\mu_B=1.
\end{equation}

Figure~\ref{fig:certificate-phase-calibration} compares this boundary with observed convergence and linear rates. Every tested certified instance converges for both couplings. The symmetric coupling also converges throughout the tested exterior, whereas the cyclic shift follows the canonical boundary more closely. Within the certified region, the median observed-to-certified diagonal-Hilbert factors are $0.617$ for $M_{\rm sym}$ and $0.943$ for $M_{\rm ns}$.

\begin{figure}[t]
\centering
\includegraphics[width=0.80\textwidth]{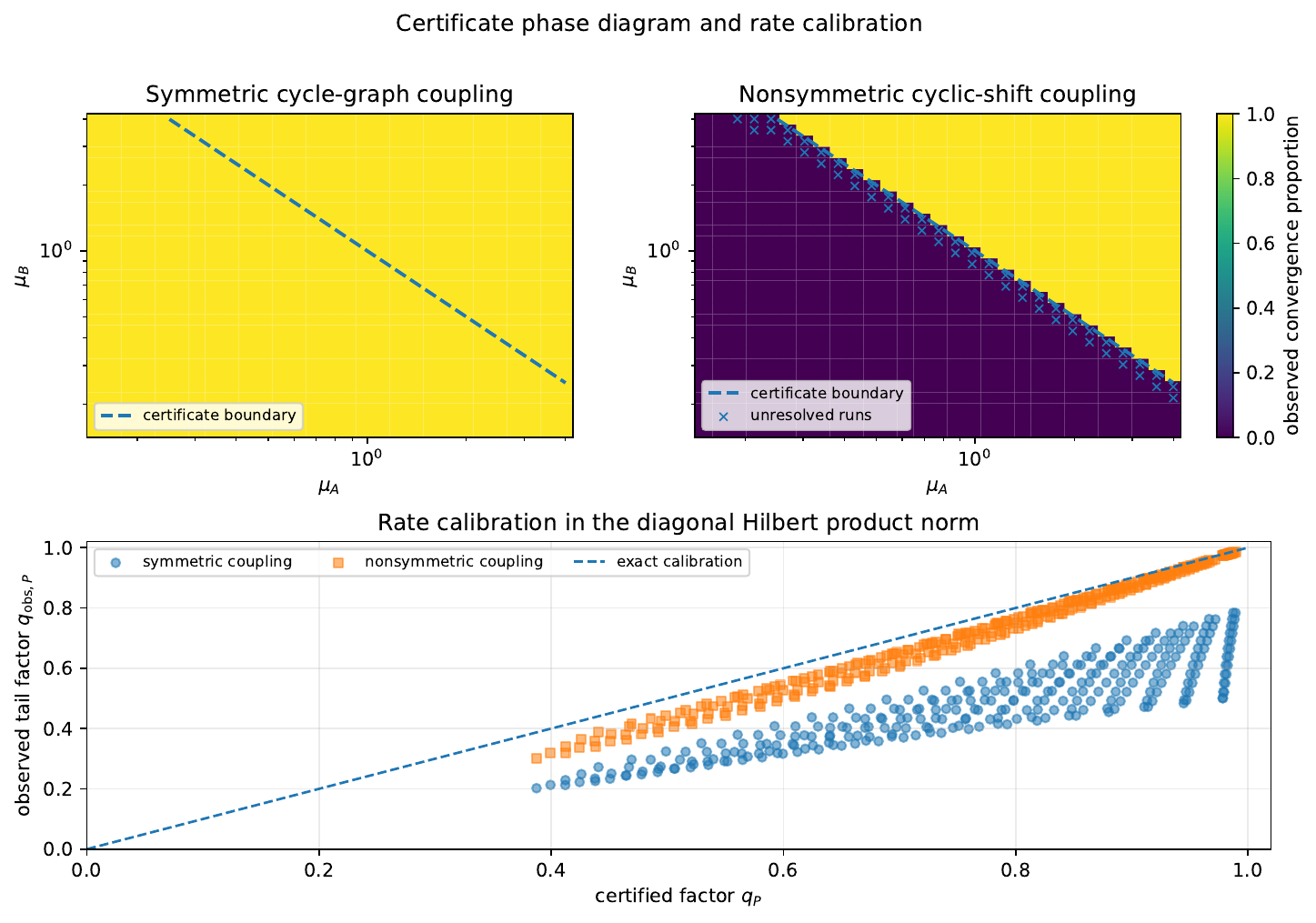}
\caption{Certificate boundary and rate calibration for two couplings with the same operator norm. The dashed curve is the common boundary \eqref{eq:certificate-boundary}; the upper panels compare it with observed convergence, while the lower panel compares certified and observed tail factors.}
\label{fig:certificate-phase-calibration}
\end{figure}

The comparison isolates the information lost when $H$ is represented only by $L_H$. Self-adjoint structure creates substantial signed and spectral cancellation, while the cyclic shift more closely approaches the absolute block-gain model.

\section{Convergence of pure compositions under weaker assumptions}
\label{sec:pure-composition}

In contrast to the coupled setting of~\Cref{sec:gain-analysis}, setting $H=b=0$ reduces the master problem to $0\in A(Bx)$ and the kernel family to \eqref{eq:pure-family}. The first row then becomes autonomous, yielding a lower-triangular dependence that sharpens the convergence analysis: joint resolvent contractivity gives a direct triangular contraction, while asymptotic convergence can be established under strong monotonicity of only one operator through a tracking argument.

\subsection{Joint resolvent contractivity}
\label{subsec:pure-two-sided}

\begin{theorem}[Global linear convergence under two-operator contractivity]
\label{thm:pure-linear}
Assume
\begin{equation}
\label{eq:pure-lipschitz}
\Lip(J_{\alpha A})\leq\ell_A<1,
\qquad
\Lip(J_{\beta B})\leq\ell_B<1.
\end{equation}
For $\theta\in[-1,1]$, set
\begin{equation}
\label{eq:pure-gain-theta}
L_{AB,\theta}:=
\begin{bmatrix}\ell_A&0\\c_\theta&\ell_B\end{bmatrix},
\qquad
c_\theta:=\beta\ell_B((1+\theta)\ell_A+|\theta|).
\end{equation}
Then
\begin{equation}
\label{eq:pure-spectral-radius}
\rho(L_{AB,\theta})=\max\{\ell_A,\ell_B\}<1.
\end{equation}
Consequently, \eqref{eq:pure-family} has a unique fixed point and converges globally and linearly in the weighted norms of \Cref{thm:global}.
\end{theorem}

\begin{proof}
For two trajectories,
\[
\norm{\Delta y^+}\leq\ell_A\norm{\Delta y},
\qquad
\norm{\Delta x^+}\leq\ell_B\norm{\Delta x}+c_\theta\norm{\Delta y}.
\]
This gives \eqref{eq:pure-gain-theta}. Since the matrix is lower triangular, its spectral radius is the larger diagonal entry; \Cref{thm:global} then applies with $L_H=0$.
\end{proof}

Strong monotonicity of both factors is a standard sufficient condition for \eqref{eq:pure-lipschitz}, with $\ell_A=(1+\alpha\mu_A)^{-1}$ and $\ell_B=(1+\beta\mu_B)^{-1}$. The three distinguished algorithms have the same diagonal spectral condition, and their comparison matrices $L_{AB,\theta}$ differ only in the off-diagonal coefficient $c_\theta$.

\begin{proposition}[Sequential transmission minimizes the certified downstream gain]
\label{prop:pure-theta-zero}
The coefficient $c_\theta$ is minimized on $[-1,1]$ at $\theta=0$, and
\begin{equation}
\label{eq:pure-endpoint-gains}
c_{-1}=\beta\ell_B,
\qquad
c_0=\beta\ell_A\ell_B,
\qquad
c_1=\beta\ell_B(2\ell_A+1).
\end{equation}
\end{proposition}

\begin{proof}
For $\theta=-t\in[-1,0]$, $c_{-t}=\beta\ell_B(\ell_A+t(1-\ell_A))$; for $\theta\in[0,1]$, $c_\theta=\beta\ell_B(\ell_A+\theta(1+\ell_A))$. Both expressions are nondecreasing away from zero.
\end{proof}

The proposition concerns the componentwise transfer certificate. Additional signed structure can further reduce the actual transient coupling.

\subsection{One-sided convergence}
\label{subsec:pure-one-sided}

\begin{theorem}[Triangular convergence under single operator strong monotonicity]
\label{thm:pure-one-sided}
Let $A,B$ be maximal monotone, let $\theta\in[-1,1]$, and assume $\mathcal S_{A,B}\neq\varnothing$.
\begin{enumerate}
\item If $A$ is $\mu_A$-strongly monotone, then $\zer A=\{y^\star\}$, $y^k\to y^\star$ linearly, and $x^k\rightharpoonup x^\star\in B^{-1}y^\star$. In finite dimensions, the convergence of $x^k$ is strong.
\item If $B$ is $\mu_B$-strongly monotone, then $y^k\rightharpoonup y^\star\in\zer A$. Whenever $y^k\to y^\star$ strongly, $B^{-1}y^\star=\{x^\star\}$ and $x^k\to x^\star$ strongly. This condition holds automatically in finite dimensions.
\end{enumerate}
\end{theorem}

\begin{proof}
Suppose first that $A$ is strongly monotone. Since $\mathcal S_{A,B}\neq\varnothing$, the set $\zer A$ is nonempty, and strong monotonicity makes its element $y^\star$ unique. Set $q_A=(1+\alpha\mu_A)^{-1}\in(0,1)$. The resolvent $J_{\alpha A}$ is $q_A$-Lipschitz, so the first row gives
\begin{equation}
\label{eq:pure-y-linear}
\norm{y^k-y^\star}\leq q_A^k\norm{y^0-y^\star}.
\end{equation}
Using $\bar y^k=(1+\theta)y^{k+1}-\theta y^k$ and \eqref{eq:pure-y-linear}, we obtain
\begin{equation}
\label{eq:pure-transmitted-error}
\begin{aligned}
\norm{\bar y^k-y^\star}
&\leq(1+\theta)\norm{y^{k+1}-y^\star}+|\theta|\norm{y^k-y^\star}\\
&\leq\bigl((1+\theta)q_A+|\theta|\bigr)q_A^k\norm{y^0-y^\star}.
\end{aligned}
\end{equation}
Thus $\sum_{k=0}^{\infty}\norm{\bar y^k-y^\star}<\infty$. Define $Qx:=J_{\beta B}(x+\beta y^\star)=J_{\beta(B-y^\star)}x$ and $C:=\Fix Q=B^{-1}y^\star$. The map $Q$ is firmly nonexpansive, and $\mathcal S_{A,B}\neq\varnothing$ implies $C\neq\varnothing$. Write the second row as $x^{k+1}=Qx^k+e^k$, where $e^k:=x^{k+1}-Qx^k$, and set $\eta_k:=\norm{e^k}$. Nonexpansiveness of $J_{\beta B}$ gives $\eta_k\leq\beta\norm{\bar y^k-y^\star}$, so \eqref{eq:pure-transmitted-error} implies $\sum_{k=0}^{\infty}\eta_k<\infty$.

For every $p\in C$, nonexpansiveness of $Q$ yields $\norm{x^{k+1}-p}\leq\norm{x^k-p}+\eta_k$. Consequently, the sequence $\norm{x^k-p}+\sum_{j=k}^{\infty}\eta_j$ is nonincreasing and nonnegative. Since its tail term tends to zero, $\norm{x^k-p}$ converges for every $p\in C$, and $(x^k)$ is bounded. Set $M_p:=\sup_{k\geq0}\norm{x^k-p}<\infty$. Firm nonexpansiveness of $Q$ gives $\norm{Qx^k-p}^2+\norm{x^k-Qx^k}^2\leq\norm{x^k-p}^2$. Combining this inequality with $x^{k+1}=Qx^k+e^k$ and $\norm{Qx^k-p}\leq M_p$, we obtain
\begin{equation}
\label{eq:pure-quasifejer-residual}
\begin{aligned}
\norm{x^{k+1}-p}^2
&\leq\norm{Qx^k-p}^2+2\norm{Qx^k-p}\eta_k+\eta_k^2\\
&\leq\norm{x^k-p}^2-\norm{x^k-Qx^k}^2+2M_p\eta_k+\eta_k^2.
\end{aligned}
\end{equation}
Since $\sum_{k=0}^{\infty}\eta_k<\infty$, the error terms $2M_p\eta_k+\eta_k^2$ are summable. Hence \eqref{eq:pure-quasifejer-residual} establishes quasi-Fej\'er monotonicity with respect to $C$. Summing \eqref{eq:pure-quasifejer-residual} and using nonnegativity of the squared distances yields $\sum_{k=0}^{\infty}\norm{x^k-Qx^k}^2<\infty$, and therefore $x^k-Qx^k\to0$.

Boundedness of $(x^k)$ ensures the existence of weakly convergent subsequences. If $x^{k_j}\rightharpoonup\widehat x$, then $x^{k_j}-Qx^{k_j}\to0$, and the demiclosedness principle gives $\widehat x\in\Fix Q=C$ \cite[Corollary~4.28]{BauschkeCombettes2017}. Thus every weak sequential cluster point belongs to $C$. Together with convergence of $\norm{x^k-p}$ for every $p\in C$, this yields $x^k\rightharpoonup x^\star\in B^{-1}y^\star$ \cite[Lemma~2.47]{BauschkeCombettes2017}. In finite dimensions, this convergence is strong.

Now suppose that $B$ is strongly monotone. Since $\zer A\neq\varnothing$, the proximal-point sequence $y^{k+1}=J_{\alpha A}y^k$ converges weakly to some $y^\star\in\zer A$. Maximality and strong monotonicity of $B$ make $B^{-1}$ single-valued and everywhere defined; set $x^\star:=B^{-1}y^\star$. Then $x^\star=J_{\beta B}(x^\star+\beta y^\star)$. With $q_B=(1+\beta\mu_B)^{-1}\in(0,1)$, the $q_B$-Lipschitz continuity of $J_{\beta B}$ gives
\begin{equation}
\label{eq:pure-x-tracking}
\begin{aligned}
\norm{x^{k+1}-x^\star}
&\leq q_B\norm{x^k-x^\star+\beta(\bar y^k-y^\star)}\\
&\leq q_B\norm{x^k-x^\star}+\beta q_B\norm{\bar y^k-y^\star}.
\end{aligned}
\end{equation}
If $y^k\to y^\star$ strongly, then $\bar y^k-y^\star=(1+\theta)(y^{k+1}-y^\star)-\theta(y^k-y^\star)\to0$. Applying \Cref{lem:vanishing-input-contraction} below to \eqref{eq:pure-x-tracking} yields $x^k\to x^\star$ strongly. In finite dimensions, the weak convergence of $y^k$ is strong, so this condition holds automatically.
\end{proof}

\begin{lemma}
\label[lemma]{lem:vanishing-input-contraction}
If nonnegative sequences $\{e_k\}$ and $\{\delta_k\}$ satisfy
$e_{k+1}\leq\kappa e_k+c\delta_k$ with $\kappa\in[0,1)$, $c\geq0$, and $\delta_k\to0$, then $e_k\to0$.
\end{lemma}

\begin{proof}
Iteration gives $e_k\leq\kappa^ke_0+c\sum_{j<k}\kappa^{k-1-j}\delta_j$. A finite-prefix/tail decomposition and the geometric-series bound give the result.
\end{proof}

\begin{remark}[Halpern regularization in the $B$-strong branch]
\label{rem:pure-halpern}
In an infinite-dimensional Hilbert space, strong convergence of the autonomous first row can be enforced by
\begin{equation}
\label{eq:pure-halpern-row}
y^{k+1}=\lambda_k u+(1-\lambda_k)J_{\alpha A}y^k,
\end{equation}
where $u\in\mathcal X$ is a fixed anchor and $(\lambda_k)_{k\ge0}\subset(0,1)$ satisfies $\lambda_k\to0$, $\sum_{k=0}^{\infty}\lambda_k=\infty$, and $\sum_{k=0}^{\infty}|\lambda_{k+1}-\lambda_k|<\infty$. Since $J_{\alpha A}$ is nonexpansive and $\operatorname{Fix}J_{\alpha A}=\zer A\ne\varnothing$, Halpern--Wittmann theory gives $y^k\to y^\star:=\proj_{\zer A}u$ strongly \cite{Halpern1967,Wittmann1992}. Hence $\bar y^k\to y^\star$, and the same tracking recursion yields $x^k\to B^{-1}y^\star$ strongly.
\end{remark}

\subsection{Example: nonlinear capacitor simulation}
\label{subsec:single-varicap-example}

This section gives a simple numerical demonstration of the pure composition algorithm with $\theta = -1$, simulating the periodic voltage response of the nonlinear varicap capacitor introduced in \Cref{subsec:varicap-model}.
For the numerical computation, we discretize the device model \eqref{eq:varicap-composition-model}--\eqref{eq:charge-mean-slice} using $N$ grid points with $h=T/N$. The periodic backward difference $D$ is one possible discretization of $d/dt$:
\begin{equation}
\label{eq:periodic-difference}
(Dq)_j:=\frac{q_j-q_{j-1}}{h},\qquad q_{-1}:=q_{N-1}.
\end{equation}
For every $z\in\R^N$,
\begin{equation}
\label{eq:periodic-difference-monotone}
\inner{z}{Dz}=\frac{1}{2h}\sum_{j=0}^{N-1}(z_j-z_{j-1})^2\geq0,\qquad D\one=0,
\end{equation}
so $D$ is monotone and its kernel is the constant mode, $D\mathbf{1}=0$. We use a zero-mean discrete current $i^\star\in\mathbb R^N$ and apply $B_\varepsilon$, defined in \eqref{eq:single-varicap-law}, componentwise. We define the discrete mean-charge slice as $\mathcal Q_{\rm ref}:=\{q\in\mathbb R^N:N^{-1}\mathbf{1}^{\mathsf T}q=\bar q_{\rm ref}\}$.

The dynamic operator $A(q)=Dq-i^\star$ is maximal monotone in finite dimensions, while the static operator $B_\varepsilon$ is maximal $\varepsilon$-strongly monotone. The second branch of \Cref{thm:pure-one-sided} therefore applies on the prescribed charge-mean slice. Because $D\mathbf{1}=0$, the current equation leaves the constant charge mode undetermined. We therefore prescribe the mean charge $\bar q_{\rm ref}$ and assume that the initial charge lies in $\mathcal Q_{\rm ref}$, equivalently $N^{-1}\mathbf{1}^{\mathsf T}q^0=\bar q_{\rm ref}$. In centered coordinates $q_b:=q-\bar q_{\rm ref}\mathbf{1}$, this is the zero-average condition $N^{-1}\mathbf{1}^{\mathsf T}q_b^0=0$.

For the original kernel $\theta=-1$, the iteration is
\begin{subequations}
\label{eq:single-varicap-iteration}
\begin{align}
\widetilde q^{k+1}
&=(I+\alpha D)^{-1}(q^k+\alpha i^\star),
\qquad
q^{k+1}=P_{\bar q_{\rm ref}}\widetilde q^{k+1},
\label{eq:single-varicap-q}\\
v^{k+1}
&=J_{\beta B_\varepsilon}(v^k+\beta q^k).
\label{eq:single-varicap-v}
\end{align}
\end{subequations}
Here $P_{\bar q_{\rm ref}}$ is the affine projector onto $\mathcal Q_{\rm ref}$. The shifted dynamic solve preserves the prescribed mean in exact arithmetic, while the projection fixes the gauge in finite precision.


\begin{figure}[t]
\centering
\begin{minipage}[c]{0.47\textwidth}
\centering
\includegraphics[width=\linewidth]{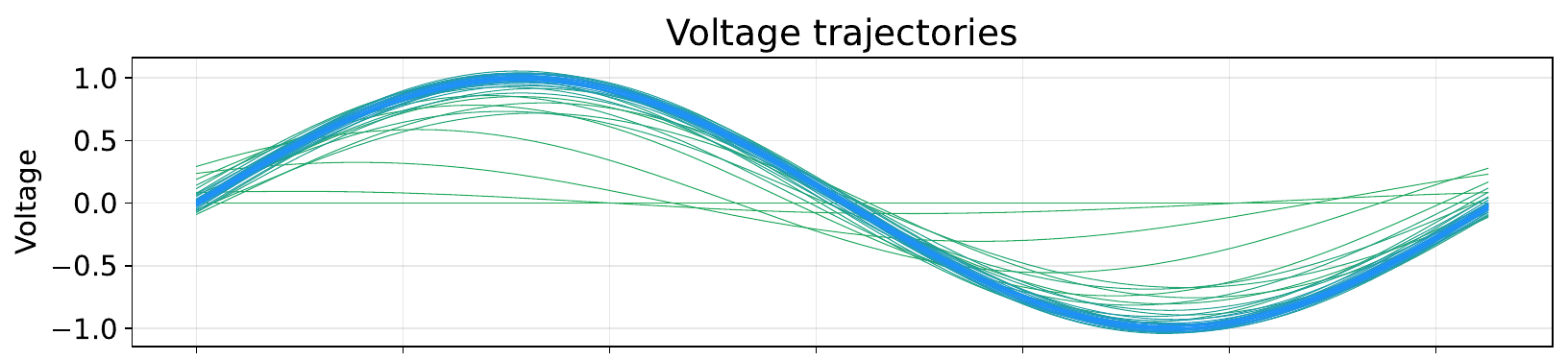}
\end{minipage}\hfill
\begin{minipage}[c]{0.47\textwidth}
\centering
\includegraphics[width=\linewidth]{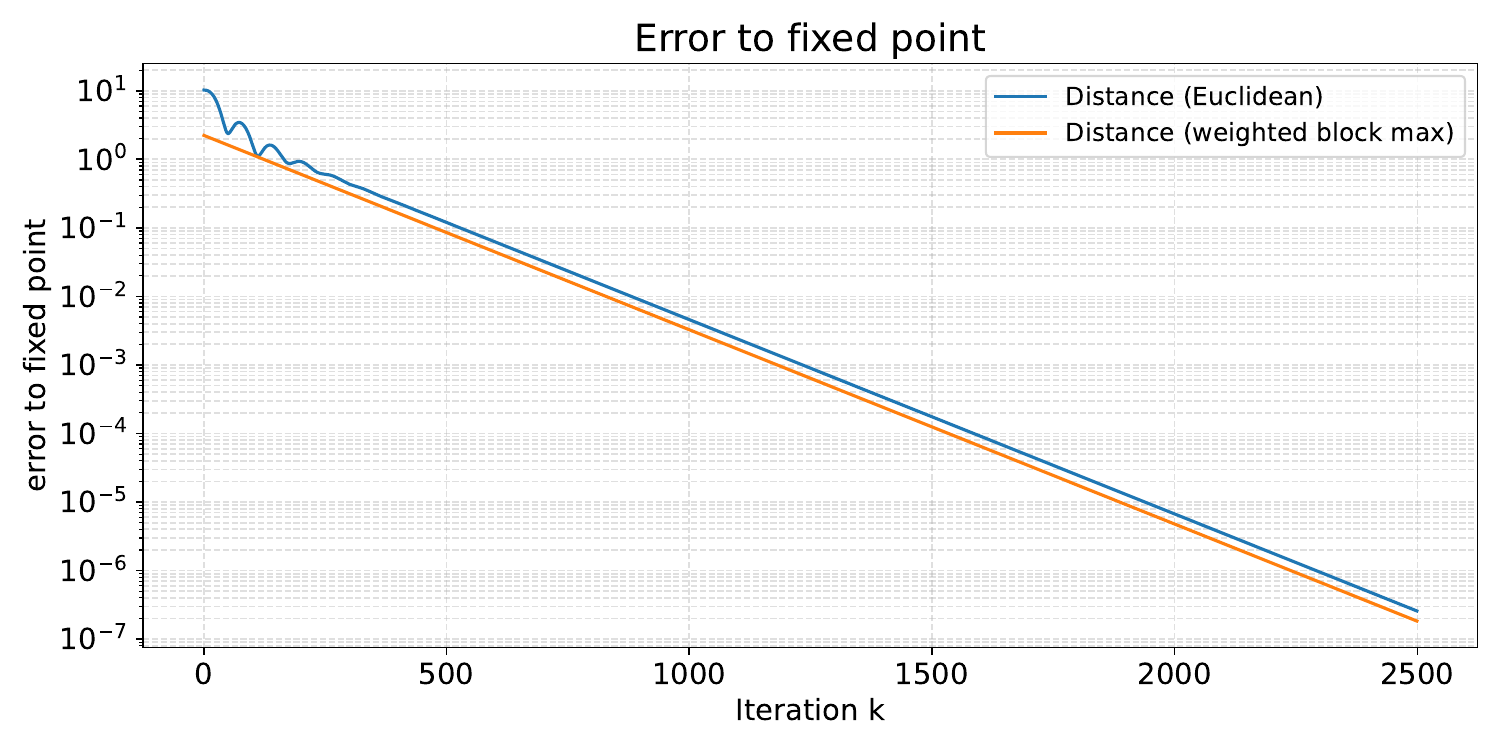}
\end{minipage}
\caption{Single-varicap voltage recovery: selected voltage trajectories and fixed-point errors.}
\label{fig:single-varicap}
\end{figure}


The experiment uses
\begin{equation}
\label{eq:single-varicap-parameters}
(N,T,K,V_0,\varepsilon,\alpha,\beta)
=(200,2\pi,1,5,10^{-6},0.1,0.1).
\end{equation}
We prescribe $v_j^\star=\sin(jT/N)$, set $q^\star=B_\varepsilon(v^\star)$, $i^\star=Dq^\star$, and $\bar q_{\rm ref}=N^{-1}\one\T q^\star$. The charge sequence converges on the affine slice. Since $J_{\beta B_\varepsilon}$ has contraction factor $(1+\beta\varepsilon)^{-1}$,
\begin{equation}
\label{eq:single-varicap-tracking}
\norm{v^{k+1}-v^\star}
\leq(1+\beta\varepsilon)^{-1}
\bigl(\norm{v^k-v^\star}+\beta\norm{q^k-q^\star}\bigr),
\end{equation}
and the voltage converges to the prescribed periodic trajectory. Appendix~\ref{app:varicap-protocol} records the Fourier solve, scalar device resolvent, initialization, and plotting protocol.

\section{Application to cascaded deep equilibrium models}
\label{sec:mondeq}

Implicit deep learning represents hidden features through fixed-point equations \cite{ElGhaouiEtAl2021}, and multiscale deep equilibrium models couple equilibrium representations across multiple feature resolutions \cite{BaiKoltunKolter2020}. Monotone operator equilibrium networks formulate implicit layers as strongly monotone inclusions and compute their equilibria using operator-splitting methods \cite{WinstonKolter2020}, and have equivalent nonlinear circuit representations \cite{Chaffey2025}. Feed-forward cascades connect the outputs of such equilibrium blocks to the inputs of subsequent blocks. The algorithms we present in this paper can be used to solve the equilibria of such cascades.

We apply the Douglas--Rachford specialization of \Cref{alg:triangular} to each block of a finite feed-forward cascade. For a fixed block input, the equilibrium inclusion can be expressed in the form $0\in Hx+A(Bx)-b$, and the associated triangular update becomes a Douglas--Rachford step under the shadow-state change of variables. We interleave these updates by advancing each block once per global cycle and immediately transmitting its primal resolvent readout to the next block. Thus the algorithm iterates the shadow states while coupling the blocks through their primal readouts.

The convergence analysis extends the lower-triangular comparison mechanism of \Cref{sec:pure-composition} to this finite block network. The resulting comparison matrix has the individual block contraction factors on its diagonal, while the feed-forward gains enter only its strict lower triangle. Under the block assumptions stated below, this structure yields global linear convergence for arbitrary fixed real interconnection gains, including cases where the assembled equilibrium inclusion is nonmonotone in the Euclidean inner product.

\subsection{Finite feed-forward cascade}
\label{subsec:mondeq-model}

Let $n\geq2$. For each block $i$, let $\mathcal Z_i$ and $\mathcal U_i$ be finite-dimensional real Hilbert spaces for the state $z_i\in\mathcal Z_i$ and input $u_i\in\mathcal U_i$, respectively. Let $\Psi_i\colon\mathcal Z_i\rightrightarrows\mathcal Z_i$ be maximal monotone, let $B_i\colon\mathcal U_i\to\mathcal Z_i$ be linear, and let the self-adjoint operator $H_i\colon\mathcal Z_i\to\mathcal Z_i$ satisfy
\begin{equation}
  0\in H_i z_i+\Psi_i(z_i)+B_i u_i,
  \qquad
  \mu_iI\preceq H_i\preceq L_iI,
  \qquad 0<\mu_i\leq L_i<\infty.
  \label{eq:mondeq-single-block}
\end{equation}
For $i=1,\ldots,n-1$, let $C_i\colon\mathcal Z_i\to\mathcal U_{i+1}$ and $D_i\colon\mathcal U_i\to\mathcal U_{i+1}$ be linear. For a prescribed external input $u_{\rm ext}$, the feed-forward interconnection is
\begin{equation}
u_1=u_{\mathrm{ext}},\qquad u_{i+1}=\sigma_i(C_i z_i+D_i u_i),\qquad i=1,\ldots,n-1,\qquad \sigma_i\in\mathbb{R}.
\label{eq:mondeq-cascade}
\end{equation}
The interconnection gains $\sigma_i$ are fixed real scalars. Since $H_i+\Psi_i$ is strongly monotone, every block has a unique equilibrium for fixed input; hence the cascaded primal equilibrium is uniquely determined by solving the blocks successively for $i=1,\ldots,n$.

\subsection{Interleaved Douglas--Rachford inference}
\label{subsec:mondeq-drs}
\label{subsec:mondeq-algorithm}
Fix an input $u\in U_i$ and a stepsize $\lambda_i>0$, and define $F_{i,u}(z):=H_i z+B_i u$. Apply the Douglas--Rachford specialization of \Cref{alg:triangular} with
\begin{equation}
\label{eq:mondeq-master-specialization}
\begin{gathered}
A=\Psi_i,\qquad B=F_{i,u}^{-1},\qquad H=I,\qquad b=0,\\
\alpha=\lambda_i,\qquad\beta=\lambda_i^{-1},\qquad\theta=1.
\end{gathered}
\end{equation}
Writing the master variables as $(y,x)=(a,d)$, the fixed-point relations become $0\in\Psi_i(a)+d$ and $d=F_{i,u}(a)$, which recover the block \eqref{eq:mondeq-single-block}. Under \eqref{eq:mondeq-master-specialization}, one triangular step reads
\begin{equation}
\label{eq:mondeq-triangular-step}
\begin{aligned}
a^+&=J_{\lambda_i\Psi_i}(a-\lambda_i d),\\
d^+&=J_{\lambda_i^{-1}F_{i,u}^{-1}}\bigl(d+\lambda_i^{-1}(2a^+-a)\bigr).
\end{aligned}
\end{equation}
Introduce the shadow state $s:=a-\lambda_i d$ and its update $s^+:=a^+-\lambda_i d^+$. The first row of \eqref{eq:mondeq-triangular-step} gives $a^+=J_{\lambda_i\Psi_i}s$, while the inverse-resolvent identity applied to the second row yields
\begin{equation}
\label{eq:mondeq-shadow-reduction}
\begin{aligned}
\lambda_i d^+&=2a^+-s-J_{\lambda_iF_{i,u}}(2a^+-s),\\
s^+&=s+J_{\lambda_iF_{i,u}}(2a^+-s)-a^+.
\end{aligned}
\end{equation}
Using $J_{\lambda_iF_{i,u}}v=J_{\lambda_iH_i}(v-\lambda_iB_i u)$ in \eqref{eq:mondeq-shadow-reduction} gives the following shadow-state formulation. For $s\in Z_i$, define
\begin{equation}
\begin{aligned}
  p_{i,u}(s)&=J_{\lambda_i\Psi_i}s,
  &q_{i,u}(s)&=J_{\lambda_iH_i}\bigl(2p_{i,u}(s)-s-\lambda_iB_i u\bigr),\\[-1mm]
  T_{i,u}(s)&=s+q_{i,u}(s)-p_{i,u}(s).
\end{aligned}
\label{eq:mondeq-drs-components}
\end{equation}
Equivalently, $T_{i,u}=\tfrac12(I+R_{\lambda_i(H_i+B_i u)}R_{\lambda_i\Psi_i})$, where $H_i+B_i u$ denotes the affine map $z\mapsto H_i z+B_i u$.  Thus $s$ is the Douglas--Rachford shadow state, while $p_{i,u}(s)$ and $q_{i,u}(s)$ are primal resolvent points \cite{BauschkeCombettes2017,EcksteinBertsekas1992}.

\begin{proposition}[Shadow fixed points and cascaded equilibria]
\label[proposition]{prop:mondeq-shadow-primal}
For fixed $u$, $s_i^\star=T_{i,u}(s_i^\star)$ if and only if
\begin{equation}
  p_{i,u}(s_i^\star)=q_{i,u}(s_i^\star)=z_i^\star,
  \qquad
  0\in H_i z_i^\star+\Psi_i(z_i^\star)+B_i u,
  \qquad
  s_i^\star=z_i^\star-\lambda_i(H_i z_i^\star+B_i u).
\label{eq:shadow-fixed-point}
\end{equation}
Consequently, the shadow reconstruction above establishes a one-to-one correspondence between solutions of \eqref{eq:mondeq-single-block}--\eqref{eq:mondeq-cascade} and fixed points of the interleaved cascade below.
\end{proposition}

\begin{proof}
At a shadow fixed point, \eqref{eq:mondeq-drs-components} gives $p=q=:z$.  The two resolvent relations yield $s-z\in\lambda_i\Psi_i(z)$ and $s-z=-\lambda_i(H_i z+B_i u)$, which are equivalent to the displayed inclusion and shadow reconstruction.  Reversing these identities proves the converse, and applying the argument successively along \eqref{eq:mondeq-cascade} proves the global statement.
\end{proof}

One global inference cycle applies one Douglas--Rachford specialization of \Cref{alg:triangular} to each block in the order $i=1,\ldots,n$. Given $\boldsymbol s^k=(s_1^k,\ldots,s_n^k)$, set $u_1^{k+1}=u_{\rm ext}$ and, for $i=1,\ldots,n$, compute
\begin{equation}
\begin{aligned}
  p_i^k&=J_{\lambda_i\Psi_i}s_i^k,
  &q_i^k&=J_{\lambda_iH_i}\bigl(2p_i^k-s_i^k-\lambda_iB_i u_i^{k+1}\bigr),\\[-1mm]
  s_i^{k+1}&=s_i^k+q_i^k-p_i^k,
  &u_{i+1}^{k+1}&=\sigma_i(C_iq_i^k+D_i u_i^{k+1}),\quad i<n.
\end{aligned}
\label{eq:mondeq-iteration}
\end{equation}
The last input row is omitted for $i=n$.  The primal readout $q_i^k$, rather than the shadow state, is transmitted immediately to block $i+1$. Denote the resulting global map by
\begin{equation}
  \boldsymbol s^{k+1}=\mathcal T_n\boldsymbol s^k.
  \label{eq:mondeq-global-map}
\end{equation}
Sequential inference instead completes each block equilibrium before initiating the next solve; \eqref{eq:mondeq-iteration} advances all blocks once per global cycle.

\subsection{Lower-triangular contraction}
\label{subsec:mondeq-convergence}

Define
\begin{equation}
\begin{aligned}
  j_i&=(1+\lambda_i\mu_i)^{-1},
  &r_i&=\max_{t\in[\mu_i,L_i]}\left|\frac{1-\lambda_i t}{1+\lambda_i t}\right|,
  &\kappa_i&=\frac{1+r_i}{2}<1,\\[-1mm]
  \eta_i&=\lambda_i\|J_{\lambda_iH_i}B_i\|,
  &g_i&=|\sigma_i|\,\|C_i\|,
  &\tau_i&=|\sigma_i|\,\|D_i\|+g_i\eta_i,\\[-1mm]
  h_{ij}&=g_jj_j\prod_{\ell=j+1}^{i-1}\tau_\ell,
  &&1\leq j<i\leq n.
\end{aligned}
\label{eq:lower-triiangular-paras}
\end{equation}
where an empty product equals one.  Let $L_n\in\mathbb R_+^{n\times n}$ be the lower-triangular matrix
\begin{equation}
  (L_n)_{ij}=
  \begin{cases}
    0,&j>i,\\
    \kappa_i,&j=i,\\
    \eta_i h_{ij},&j<i.
  \end{cases}
  \label{eq:mondeq-comparison-matrix}
\end{equation}

\begin{theorem}[Global linear convergence of the interleaved cascade]
\label{thm:mondeq-convergence}
Under \eqref{eq:mondeq-single-block}, the global map satisfies
\begin{equation}
  e(\mathcal T_n\boldsymbol s-\mathcal T_n\widetilde{\boldsymbol s})
  \leq L_n e(\boldsymbol s-\widetilde{\boldsymbol s}),
  \qquad
  \rho(L_n)=\max_{1\leq i\leq n}\kappa_i<1.
  \label{eq:mondeq-comparison}
  \end{equation}
Hence $\mathcal T_n$ has a unique fixed point and \eqref{eq:mondeq-global-map} converges globally and linearly in a Perron-weighted block maximum norm.  The primal readouts and feed-forward inputs converge linearly to the unique cascaded equilibrium.
\end{theorem}

\begin{proof}
Nonexpansiveness of $R_{\lambda_i\Psi_i}$ and the spectral bounds on $H_i$ give, for arbitrary inputs $u,\widetilde u$,
\begin{equation}
\begin{aligned}
  \|q_{i,u}(s)-q_{i,\widetilde u}(\widetilde s)\|
  &\leq j_i\|s-\widetilde s\|+\eta_i\|u-\widetilde u\|,\\
  \|T_{i,u}(s)-T_{i,\widetilde u}(\widetilde s)\|
  &\leq \kappa_i\|s-\widetilde s\|+\eta_i\|u-\widetilde u\|.
\end{aligned}
\label{eq:mondeq-block-contraction}
\end{equation}
The interconnection therefore satisfies
\(
\|u_{i+1}-\widetilde u_{i+1}\|
\leq g_i j_i\|s_i-\widetilde s_i\|+\tau_i\|u_i-\widetilde u_i\|.
\)
Since $u_1=\tilde u_1$, induction on $i$ yields
\(
\|u_i-\widetilde u_i\|\leq\sum_{j<i}h_{ij}\|s_j-\widetilde s_j\|,
\)
which substituted into \eqref{eq:mondeq-block-contraction} gives the comparison matrix in \eqref{eq:mondeq-comparison}.  Its eigenvalues are the diagonal entries $\kappa_i$; \Cref{lem:perron-mblock} then gives the weighted contraction and Banach--Picard convergence. \Cref{prop:mondeq-shadow-primal} identifies the limiting primal equilibrium.
\end{proof}

The feed-forward interconnection places every inter-block comparison coefficient strictly below the diagonal of $L_n$. This extends the lower-triangular comparison structure of the pure-composition family in \Cref{sec:pure-composition} to finitely many blocks. The inter-block coefficients affect the contraction weights and the propagation of perturbations, while $\rho(L_n)=\max_i\kappa_i$ depends only on the individual Douglas--Rachford contraction factors. For $n=2$,
\begin{equation}
  L_2=\begin{pmatrix}
\kappa_1 & 0 \\
\lambda_2\|J_{\lambda_2 H_2}B_2\|\,|\sigma_1|\,\|C_1\|j_1 & \kappa_2
\end{pmatrix},\qquad \rho(L_2)=\max\{\kappa_1,\kappa_2\}<1.
  \label{eq:mondeq-cascade-gain}
\end{equation}

\begin{remark}[A nonmonotone ReLU cascade]
\label{rem:nonmonotone-relu-cascade}
Consider the cascade \eqref{eq:mondeq-single-block}--\eqref{eq:mondeq-cascade} with $n=2$ scalar blocks, state variables $z_i\in\mathcal Z_i=\mathbb{R}$, and inputs $u_i\in\mathcal U_i=\mathbb{R}$ for $i=1,2$. Choose $H_1=H_2=I$, $\Psi_1=\Psi_2=N_{[0,\infty)}$, $B_1=B_2=-I$, $C_1=I$, $D_1=0$, $u_{\mathrm{ext}}=a>0$, and $\sigma_1=\sigma\in\mathbb{R}$. Here $N_C$ denotes the normal cone to $C$, so that $J_{\Psi_i}=\operatorname{ReLU}$. The operators $\Psi_i$ are maximal monotone, and the spectral bounds in \eqref{eq:mondeq-single-block} hold with $\mu_i=L_i=1$. Equations  \eqref{eq:mondeq-single-block}--\eqref{eq:mondeq-cascade} reduce to
\[
0\in z_1-a+N_{[0,\infty)}(z_1),\qquad 0\in z_2-\sigma z_1+N_{[0,\infty)}(z_2),
\]
with unique equilibrium $z^\star=(a,\max\{\sigma a,0\})$. The assembled inclusion is $0\in\mathcal{G}_\sigma(z)$, where
\[
\mathcal{G}_\sigma(z)=M_\sigma z-\begin{pmatrix}a\\0\end{pmatrix}+N_{[0,\infty)^2}(z),\quad M_\sigma=\begin{pmatrix}1&0\\-\sigma&1\end{pmatrix},\quad \lambda_{\min}\!\left(\frac{M_\sigma+M_\sigma^{\mathsf T}}{2}\right)=1-\frac{|\sigma|}{2}.
\]
Since the normal cone vanishes on $(0,\infty)^2$, $\mathcal{G}_\sigma$ is not monotone in the Euclidean inner product when $|\sigma|>2$.

Choose $\lambda_1=\lambda_2=1$. The Douglas--Rachford specialization \eqref{eq:mondeq-master-specialization} of \Cref{alg:triangular}, interleaved according to \eqref{eq:mondeq-iteration}, gives the explicit shadow map
\[
T_2(s_1,s_2)=\left(\frac{s_1+a}{2},\frac{s_2}{2}+\frac{\sigma}{4}(|s_1|+a)\right).
\]
The constants in \eqref{eq:lower-triiangular-paras} satisfy $j_i=\eta_i=\kappa_i=1/2$, so the two-block comparison matrix \eqref{eq:mondeq-cascade-gain} becomes
\[
L_2=\begin{pmatrix}\tfrac12&0\\\tfrac{|\sigma|}{4}&\tfrac12\end{pmatrix},\qquad \rho(L_2)=\frac12<1.
\]
\Cref{thm:mondeq-convergence} therefore guarantees global linear convergence of the interleaved iteration \eqref{eq:mondeq-iteration} for every fixed real $\sigma$. \Cref{prop:mondeq-shadow-primal} and the reconstruction formula \eqref{eq:shadow-fixed-point} identify its shadow limit as $s^\star=(a,\sigma a)$ and its limiting primal readouts as $z^\star$. In particular, the convergence guarantee applies when $|\sigma|>2$, even though the assembled equilibrium inclusion is nonmonotone in the Euclidean inner product.
\end{remark}

\section*{Conclusions and future work}
\addcontentsline{toc}{section}{Conclusions}

A broad class of coupled convex optimisation equilibria, bilinear saddle and potential systems, and nonlinear electronic-device laws can be written as $0\in Hx+A(Bx)-b$. Retaining the intermediate graph variable separates the two monotone factors and exposes a directed lift. The proposed kernel family exploits this lifted representation to yield a lower-triangular warped-resolvent realization evaluated by the ordinary resolvents of $A$ and $B$.

The kernel parameter organizes both established and new algorithms without changing the lifted fixed-point set. Forward--backward, inertial, primal--dual, and Douglas--Rachford methods arise from exact specializations, while the pure composition admits original, sequential, and reflected factorwise iterations. This separation between problem data and kernel data makes the algorithmic role of each construction explicit.

The convergence analysis in \Cref{subsec:global-contraction} is governed by a nonnegative comparison matrix. Its Schur condition gives constructive weighted norms, global linear convergence, and residual error bounds. The phase analysis in \Cref{subsec:phase-attainability} shows sharpness of the convergence conditions for nonpositive $\theta$ and a strict gap for some operator classes when $\theta > 0$. In the pure-composition setting, the autonomous first row allows us to show convergence under strong monotonicity of a single operator only (\Cref{subsec:pure-one-sided}). This is illustrated numerically with periodic voltage recovery for a nonlinear varicap (\Cref{subsec:single-varicap-example}). The MonDEQ cascade extends the same mechanism to a finite lower-triangular network, where feed-forward couplings affect transient weights while the comparison-system spectral radius is determined by the diagonal block factors (\Cref{subsec:mondeq-convergence}).

A natural direction for further research is a systematic characterization of the Hilbert geometries and kernel geometries associated with coupled operator compositions. This includes block-coupled and variable metrics beyond $P_t=\diag(I,tI)$, sharp conditions under which a triangular warped resolvent is averaged through monotonicity of the lift or of an operator--kernel pair, and extensions to nonlinear kernels. A complementary systems-theoretic direction is to relate the bilinear potential specialization to reciprocity and storage-function structure in dissipative systems \cite{Willems1972PartII}, and to potential representations of interconnected dynamical systems.

\section*{Acknowledgements}

The authors are grateful for useful discussions with Pontus Giselsson, Sebastian Banert and Richard Pates. The authors thank Professor Li Qiu (The Chinese University of Hong Kong, Shenzhen) for insightful discussions on phase analysis and the choice of metrics, which helped shape the perspectives developed in this work.

During the preparation of this work, the authors used ChatGPT (OpenAI) to assist with language editing, manuscript organization, LaTeX typesetting, and the checking of selected mathematical arguments. The authors reviewed and verified all AI-assisted content retained in the manuscript. The authors assume responsibility for all content.

\appendix
\section{Reusable finite-block gain constructions}
\label[appendix]{app:gain}

The proof of \Cref{thm:global} contains the complete two-block argument. This appendix extends the Perron-weighted block-maximum and diagonal Hilbert-product constructions to arbitrary finite block gains and identifies the optimal coefficient obtainable by diagonal rescaling of the component norms.

\subsection{Perron scaling for a finite block gain}

Let $\cX_\Pi=\cX_1\times\cdots\times\cX_m$ and define
\[
e(z)=(\norm{z_1},\ldots,\norm{z_m})\T.
\]
For $w\in\Rpp^m$, set $\norm{z}_{w,\infty}=\max_i\norm{z_i}/w_i$.

\begin{lemma}[Perron-weighted block maximum estimate]
\label[lemma]{lem:perron-mblock}
Let $T$ be a block map on $\Omega\subset\cX_\Pi$, and suppose
\begin{equation}
\label{eq:mblock-gain}
e(Tz-T\widetilde z)\leq L e(z-\widetilde z),
\qquad z,\widetilde z\in\Omega,
\end{equation}
for a nonnegative matrix $L\in\Rplus^{m\times m}$. If $\rho(L)<\gamma$, then
\[
w_\gamma=(\gamma I-L)^{-1}\one>0
\]
and
\[
\norm{Tz-T\widetilde z}_{w_\gamma,\infty}
\leq
\gamma\norm{z-\widetilde z}_{w_\gamma,\infty}.
\]
When $\gamma<1$, $T$ is a contraction in an equivalent complete norm on every closed invariant subset of $\Omega$.
\end{lemma}

\begin{proof}
The Neumann series gives $(\gamma I-L)^{-1}\geq0$. Thus $w_\gamma>0$ and $Lw_\gamma=\gamma w_\gamma-\one<\gamma w_\gamma$. If $\delta=\norm{z-\widetilde z}_{w_\gamma,\infty}$, then $e(z-\widetilde z)\leq w_\gamma\delta$, and \eqref{eq:mblock-gain} gives the result.
\end{proof}

\subsection{A constructive diagonal Hilbert product norm}

\begin{lemma}[Diagonal quadratic Lyapunov weight]
\label{lem:diagonal-hilbert}
Let $L\geq0$ and $\rho(L)<1$. Define
\begin{equation}
\label{eq:mblock-hilbert-weight}
r=(I-L)^{-1}\one,
\qquad
s=(I-L\T)^{-1}\one,
\qquad
P=\diag\paren{\frac{s_i}{r_i}}_{i=1}^m,
\end{equation}
and
\[
\eta_r=\max_i\frac{(Lr)_i}{r_i},
\qquad
\eta_s=\max_i\frac{(L\T s)_i}{s_i},
\qquad
q=\sqrt{\eta_r\eta_s}.
\]
Then $0\leq q<1$ and
\begin{equation}
\label{eq:quadratic-lyapunov}
L\T P L\preceq q^2P.
\end{equation}
Consequently, every block map satisfying \eqref{eq:mblock-gain} obeys
\[
\norm{Tz-T\widetilde z}_{P,2}
\leq
q\norm{z-\widetilde z}_{P,2},
\qquad
\norm{z}_{P,2}^2=\sum_i\frac{s_i}{r_i}\norm{z_i}^2.
\]
\end{lemma}

\begin{proof}
Schur stability gives $r,s>0$, while $Lr=r-\one<r$ and $L\T s=s-\one<s$ imply $\eta_r,\eta_s<1$. For $d\in\Rplus^m$, weighted Cauchy--Schwarz gives
\[
(Ld)_i^2\leq(Lr)_i\sum_jL_{ij}\frac{d_j^2}{r_j}.
\]
Multiplying by $s_i/r_i$ and summing yields
\[
(Ld)\T P(Ld)
\leq
\eta_r\sum_j\frac{(L\T s)_j}{r_j}d_j^2
\leq
\eta_r\eta_s d\T P d.
\]
For arbitrary $d\in\R^m$, $|Ld|\leq L|d|$ and the preceding estimate applied to $|d|$ give $(Ld)\T P(Ld)\leq\eta_r\eta_s d\T P d=q^2d\T P d$, hence $L\T PL\preceq q^2P$. Setting $d=e(z-\widetilde z)$ and using \eqref{eq:mblock-gain} gives the stated block-map estimate.
\end{proof}

\subsection{Best diagonal block scaling}

For $w\in\Rpp^m$, define
\begin{equation}
\label{eq:cw}
c(w):=\max_i\frac{(Lw)_i}{w_i}.
\end{equation}
This is the contraction coefficient certified by $L$ in the weighted block maximum norm.

\begin{proposition}[Best diagonal block scaling]
\label{prop:best-diagonal-scaling}
For every nonnegative matrix $L$,
\begin{equation}
\label{eq:collatz-wielandt}
\inf_{w>0}c(w)=\rho(L).
\end{equation}
If $L$ is irreducible, the infimum is attained by a positive Perron vector.
\end{proposition}

\begin{proof}
Equation \eqref{eq:collatz-wielandt} is the Collatz--Wielandt characterization of the Perron root. For reducible $L$, apply the irreducible result to $L+\varepsilon\one\one\T$ and let $\varepsilon\downarrow0$.
\end{proof}

Thus $\rho(L)$ is the smallest coefficient available from diagonal rescaling of the component norms. The explicit weight $(\gamma I-L)^{-1}\one$ realizes every coefficient $\gamma>\rho(L)$ without requiring irreducibility.

\section{Reproducibility details for the numerical illustrations}
\label{app:numerics}

This appendix records the protocols for the phase-resolved and nonsmooth calibrations in \Cref{sec:certificate-calibration}, together with the implementation details of the pure-composition illustration in \Cref{sec:pure-composition}. The corresponding generated figures are supplied with the source package.

\subsection{Section 4 calibration protocols}
\label{app:optimization-protocols}

\paragraph{Phase-resolved sharpness}
For the family in \eqref{eq:phase-rotation-family}, the full real iteration matrix is a direct sum of identical $4\times4$ blocks. The scan uses $361$ equally spaced phases on $[0,\pi]$ for $\theta\in\{-1,0,1/2,1\}$. Direct linear recursions at $\varphi/\pi\in\{0,1/4,1/2,3/4,1\}$ are initialized from $(1,-0.7,0.4,1.2)$ in one rotation block and serve as an independent check of the modal spectral calculation.

\paragraph{Nonsmooth certificate calibration}
In \Cref{subsec:certificate-calibration}, each of $\mu_A$ and $\mu_B$ takes $29$ logarithmically spaced values on $[0.15,4]$. For each parameter pair and each coupling in \eqref{eq:certificate-couplings}, eight fixed initializations are used. The coordinates of $x^\star$ and $y^\star$ have independent Rademacher signs and magnitudes sampled uniformly from $[0.5,1]$, and the initial state is $y^0=y^\star+2\xi_y$ and $x^0=x^\star+2\xi_x$, where $\xi_y$ and $\xi_x$ are independent standard Gaussian vectors. The centers in \eqref{eq:certificate-exact-centers} are formed separately for every parameter pair, so $(y^\star,x^\star)$ satisfies the lifted inclusion exactly.

The relative pair error is
\[
E_k
:=
\frac{\sqrt{\norm{y^k-y^\star}^2+\norm{x^k-x^\star}^2}}
{\max\{1,\sqrt{\norm{y^\star}^2+\norm{x^\star}^2}\}}.
\]
A run is classified as convergent when $E_k\leq10^{-8}$, numerically divergent if a nonfinite value occurs or if $E_k>10^6\max\{1,E_0\}$, and unresolved if neither event occurs within $5000$ iterations. At every certified grid point, the diagonal matrix $P$ and factor $q_P$ are those supplied by \Cref{thm:global}. The observed factor $q_{\rm obs,P}$ is obtained by a least-squares fit to the semilogarithmic tail of $\norm{(y^k-y^\star,x^k-x^\star)}_{P,2}$ after excluding the initial transient and the roundoff plateau.

\subsection{Single-varicap implementation}
\label{app:varicap-protocol}

The periodic shifted solve $(I+\alpha D)^{-1}$ is evaluated in Fourier coordinates, and the prescribed charge mean is restored after every shifted solve. The scalar resolvent of $B_\varepsilon$ is evaluated componentwise by solving $u+\beta B_\varepsilon(u)=r$ on the branch $u<V_0$, together with the endpoint extension in \eqref{eq:single-varicap-law}. The voltage is initialized at zero, the initial charge is projected onto $\mathcal Q_{\rm ref}$, and the run uses $2500$ iterations. Voltage snapshots are recorded at $k=1,10,20,\ldots,400$. The two curves in the right panel of \Cref{fig:single-varicap} are errors to the terminal fixed point measured in the Euclidean product norm and in a fixed weighted block-maximum diagnostic.

\FloatBarrier
\bibliographystyle{siamplain}
\bibliography{references}

\end{document}